\documentclass[12pt,a4paper,reqno]{article}
\usepackage{graphicx}
\usepackage{amsmath}
\usepackage{amsthm}
\usepackage{amsfonts}
\usepackage{latexsym}
\usepackage{amssymb}

\newtheorem{theorem}{Theorem}[section]
\theoremstyle{plain}

\newtheorem{definition}[theorem]{Definition}

\newtheorem{proposition}[theorem]{Proposition}

\numberwithin{equation}{section}

\begin{document}
\title{Long-time asymptotic solution structure of Camassa-Holm equation
subject to an initial condition with non-zero reflection coefficient of the
scattering data.}
\author{Chueh-Hsin Chang$^{1}$, Ching-Hao Yu$^{2}$ and Tony Wen-Hann Sheu$^{1,3,4}$\\
%EndAName
{\small $^1$ \textit{Center of Advanced Study in Theoretical
Sciences (CASTS),} }\\
{\small \textit{National Taiwan University, Taipei, Taiwan} }\\
{\small $^2$ \textit{Department of Ocean Science and Engineering,} }\\
{\small \textit{Zhejiang University, Zhejiang, People's Republic of China.} }\\
{\small $^3$ \textit{Department of Engineering Science and Ocean
Engineering,} }\\
{\small \textit{National Taiwan University, Taipei, Taiwan} }\\
{\small $^4$ \textit{Institute of Applied Mathematical Science,} }\\
{\small \textit{National Taiwan University, Taipei, Taiwan} }\\
{\small $^*$ Corresponding author, twhsheu@ntu.edu.tw }}
\date{}
\maketitle

\begin{abstract}
In this article we numerically revisit the long-time solution behavior of
the Camassa-Holm equation $%
u_{t}-u_{xxt}+2u_{x}+3uu_{x}=2u_{x}u_{xx}+uu_{xxx}.$ The finite difference
solution of this integrable equation is sought subject to the newly derived
initial condition with Delta-function potential. Our underlying strategy of
deriving a numerical phase accurate finite difference scheme in time domain
is to reduce the numerical dispersion error through minimization of the
derived discrepancy between the numerical and exact modified wavenumbers.
Additionally, to achieve the goal of conserving Hamiltonians in the
completely integrable equation of current interest, a
symplecticity-preserving time-stepping scheme is developed. Based on the
solutions computed from the temporally symplecticity-preserving and the
spatially wavenumber-preserving scheme, the long-time asymptotic CH solution
characters can be accurately depicted in distinct regions of the space-time
domain featuring with their own quantitatively very different solution
behaviors. We also aim to numerically confirm that in the two transition
zones their long-time asymptotics can indeed be described in terms of the
theoretically derived Painlev\'{e} transcendents. Another attempt of this
study is to numerically exhibit a close connection between the presently
predicted finite-difference solution and the solution of the Painlev\'{e}
ordinary differential equation of type II in two different transition
zones.\bigskip

\noindent \textit{2000 Mathematics Subject Classification.}\newline
\noindent {\ }\textit{Key words and phrases. }Camassa-Holm equation;
Delta-function; initial scattering data; symplecticity-preserving;
wavenumber-preserving; long-time asymptotics; Painlev\'{e} transcendent.
\end{abstract}

% Ansatze is plural

%\vspace{5mm}

\section{Introduction}

The well known Camassa-Holm (CH) equation given below has been derived from
the simplified Navier-Stokes equation to model the motion of an inviscid
fluid flow under the influence of gravity.%
\begin{equation}
u_{t}-u_{xxt}+2u_{x}+3uu_{x}=2u_{x}u_{xx}+uu_{xxx}.  \label{eq:CH}
\end{equation}%
The solution of this shallow water equation is normally sought subject to a
fast decaying initial solution given as%
\begin{equation}
u(x,t=0)=f(x),  \label{eq:bc1_}
\end{equation}%
and%
\begin{equation}
f(x)\rightarrow 0\text{ as }|x|\rightarrow \infty .  \label{eq:bc2}
\end{equation}

Camassa-Holm equation possesses many rich mathematical properties and, as a
result, has been intensively studied since its derivation in \cite{CH1993}.
Later on, authors in \cite{ConstantinLannes2009}, \cite{Ionescu-Krus} and
\cite{Johnson2002} derived CH equation by different approaches. This shallow
water equation can be expressed in terms of different Hamiltonian functions
and permits an infinite number of conservation laws. The completely
integrable equation (\ref{eq:CH}) has multi-symplectic structure as well.
One needs therefore to numerically retain these striking structure
preserving properties while solving the CH equation in time domain by finite
difference method.

Aside from the necessity of developing a symplecticity-preserving temporal
scheme to conserve Hamiltonians in CH equation, it is also essential to
develop a dispersion error reducing CH scheme to approximate spatial
derivatives. Through the minimization of discrepancy between the exact and
numerical modified wavenumbers, we can apply the physically correct and
numerically accurate scheme to retain the distinguished nature of
Camassa-Holm solution obtained even at a large time. Let $w=u-u_{xx}+1$
denotes the momentum. Boutet de Monvel et al. \cite{MKST}, \cite%
{bib:Monvel(2010)} theoretically showed that for an initial condition $%
u\left( x,0\right) $ satisfying (\ref{eq:bc1_}), (\ref{eq:bc2}) and%
\begin{equation}
w\left( x,0\right) >0,\text{ \ }\forall x\in
%TCIMACRO{\U{211d} }%
%BeginExpansion
\mathbb{R}
%EndExpansion
,  \label{eq:bc2copy}
\end{equation}%
\begin{equation}
w\left( x,0\right) \in \left \{ \left. v\in H^{3}\left(
%TCIMACRO{\U{211d} }%
%BeginExpansion
\mathbb{R}
%EndExpansion
\right) \right \vert \int_{%
%TCIMACRO{\U{211d} }%
%BeginExpansion
\mathbb{R}
%EndExpansion
}\left( 1+\left \vert x\right \vert \right) ^{1+l}\left( \left \vert v\left(
x\right) -\kappa \right \vert +\left \vert v^{\prime }\left( x\right) \right
\vert +\left \vert v^{\prime \prime }\left( x\right) \right \vert \right)
dx<\infty \right \} \text{,}  \label{restriction2}
\end{equation}%
the asymptotic CH solution in the whole domain can be classified into a
soliton region, two oscillatory regions, one fast decay region, and two
transition regions expressed in terms of the scattering data corresponding
to the initial condition by the Riemann-Hilbert approaches. From their
important works, we are aware of the fact that two transition regions can be
modeled by the second type of Painlev\'{e} transcendents. Painlev\'{e}
equations have been regarded as the most important nonlinear ordinary
differential equations since these equations have close relevance to many
fields of physical significance \cite{bib:Iwasaki(1991)}. We are therefore
motivated to explore in more detail the connection between the solutions of
Painlev\'{e} equation (ordinary differential equation) and the CH equation
(partial differential equation) in the transition regions of $\left(
x,t\right) .$ Besides the asymptotic solutions, the decay rates of the error
between the theoretical and asymptotic solutions can be also obtained.
However, we still need to answer the following two naturally arisen
questions:\smallskip

(Q1) How long it takes for us to get a good approximation between the
long-time asymptotic solution and the analytical solution?\smallskip

(Q2) How the difference between the solutions of CH equation and its
long-time asymptotics is decayed with time?\medskip

In order to compare the numerical solutions with the asymptotic solutions,
we need to know the explicit form of the scattering data corresponding to
the initial condition. Then the asymptotic solutions can be computed
numerically. How to choose a suitable initial condition with a known
scattering data is therefore the important starting point for (Q1) and (Q2).

In the literature all the scattering data corresponding to the isospectral
problems of the CH equation were known a priori (see, for example, \cite%
{Constantin2001,JohnsonPRSLA2002,Lenells2002,LiZhang}). This means that the
number of discrete eigenvalues is assumed to be known and the reflection
coefficient has been written in an abstract form. Based on the direct and
inverse scattering theories (see, for example, \cite{DrazinJohnson}, \cite%
{Constantin2001}), we can construct an initial solution which corresponds to
the Schr\"{o}dinger operator of the Lax pair with delta function-like
potential (see Section \ref{sec: initial condition} and \cite{changsheu}).
Then the reflection coefficient of the scattering data is non-zero and only
one discrete eigenvalue exists \cite{DrazinJohnson}, \cite{Whitham}. The
detailed information about the physical properties of the scattering data
can be seen in \cite{changsheu}. Under the specified initial condition $%
u\left( x,0\right) $, the asymptotic solutions can be expressed in terms of
the parameters appeared in $u\left( x,0\right) $ (see Appendix A).

For addressing the first long-time asymptotics issue (Q1), in this study the
distinct regions (soliton region and decaying modulated region) with
qualitatively different solution natures are numerically explored in great
detail. In the zones between the soliton, two slowly decaying modulated
oscillation regions, and the fast decaying region, their solution behaviors
shall be carefully examined as well. In the two transition zones, the FDTD
solution obtained from the partial differential equation will be compared
with the solution obtained from the ordinary differential equation known as
the Painlev\'{e} equation of type II. In the light of the previous results,
we have known that after a sufficiently long wave propagation time $T$, the
numerical and asymptotic solutions become close to each other for $t\geq T$.
However, what do we mean by this long time $T$ remains quantitatively
unclear. We attempt to answer (Q1) by accurately finding the time $T$ such
that the $L^{2}$ norm difference between the numerical and asymptotic
solutions in each region becomes small enough and no significant change will
happen for $t\geq T.$ The profiles of the numerical and asymptotic solutions
in the qualitatively different six regions in space-time domain at different
times are also plotted. The two solutions indeed match very well as time
becomes large enough.

The question (Q2) will be answered from two different points of view. First,
the rate of convergence for the difference between the numerical and
asymptotic solutions will be computed. Secondly, the order of the time-decay
estimate from \cite{MKST} will be found. The exact value of powers of $t$
among the time-decay estimates remains unknown in \cite{MKST}. Let $T^{\ast
} $ be the maximum of the times found in each region in (Q1). We will use
the $\sup $ norm difference between the numerical solution and asymptotic
solution to compute the order of the time-decay estimate at $t=T^{\ast }$ in
each region. We will see that these orders obtained numerically are indeed
in the range indicated in \cite{MKST}. However, these results depend on how
large the $\left( x,t>0\right) $ plane is for the approximation between the
analytical and asymptotic solutions. If we intend to get a precise
approximation in a larger region, the time needed to reach the asymptotics
shall be longer accordingly. The details can be seen in Subsection \ref%
{(Q1)(Q2)}.

Among the numerical approaches employed to study the long-time asymptotics
of integrable systems, Boutet de Monvel et al. \cite{BdMKSZheng2010} have
studied the initial-boundary value problems of the nonlinear Schr\"{o}dinger
equation. Our point of view to be conveyed, to the authors' knowledge, seems
to be new in the sense that in the literature no similar result has been
reported before. The numerical and asymptotic solutions in each region are
compared. To achieve this goal, a specified initial condition is needed and
the corresponding scattering data must be known explicitly. On this basis,
the long-time asymptotics must be re-expressed according to the specified
initial condition.\smallskip

The rest of this article is organized as follows. In Section \ref%
{sec:equation}, some of the rich properties in the CH equation which are
useful to derive the initial scattering data in Section \ref{sec: initial
condition} and which are necessarily to be taken into account to develop a
mathematically rigorous numerical scheme in Section \ref{sec:Scheme} are
given. For conducting a long-time asymptotic analysis on the CH equation, a
new smooth initial data is derived in Section \ref{sec: initial condition}.
In this study we avoid approximating the space-time mixed derivative terms
and the third-order dispersive term by adopting the two-step solution
algorithm described in Section 4.1. In the transformed equations, the
first-order spatial derivative term is approximated by the sixth-order
accurate dispersion-error reducing scheme developed in Section 4.2.1. The
temporal derivative term is approximated by the sixth-order accurate
symplecticity preserving scheme in Section 4.2.2. In Section 5, the
numerically predicted long-time asymptotic solutions in different wave
regions are discussed in detail. Comparison of the finite difference
solution with the asymptotic solution derived from the underlying theory in
\cite{bib:Monvel(2010)} and \cite{MKST} is made. We also address the subtle
change of the solutions within the transitional regions in the result
section. The newly defined powers of $t$ for the time-decay estimates
between asymptotic and FDTD solutions in some regions are computed in this
section. Finally, some concluding remarks are drawn in Section 6 according
to the results computed from the currently developed symplecticity and
modified wavenumber preserving scheme.

\section{Some mathematical properties of the CH equation \label{sec:equation}%
}

When investigating the long-time asymptotics of the integrable CH equation,
one needs to get its corresponding Lax pair equations. Equation (\ref{eq:CH}%
) can be mathematically expressed as the compatibility condition between the
Lax pair equations consisting of a system of linear eigenvalue equations
given by%
\begin{equation}
\frac{1}{w}\left( -\psi _{xx}+\frac{1}{4}\psi \right) =\lambda \psi ,
\label{Lax1}
\end{equation}%
\begin{equation}
\psi _{t}=-\left( \frac{1}{2\lambda }+u\right) \psi _{x}+\frac{1}{2}%
u_{x}\psi .  \label{Lax2}
\end{equation}%
In this set of Lax pair equations, $w(=u-u_{xx}+1)$ and $\lambda $ denote
the momentum and eigenvalue, respectively. Steady Lax equation (\ref{Lax1})
formulates the spectral problem for the CH equation (\ref{eq:CH}) defined in
the space-time coordinate. This problem can be solved via the inverse
scattering method. In the following some useful properties for the
derivation of initial data are briefly summarized.

Note that $w\left( x,t\right) >0$ provided that equation (\ref{eq:bc2copy})
holds \cite{Constantin2001}. From \cite{MKST} and \cite{Constantin2001}, by
virtue of the Liouville transformation given below%
\begin{equation}
\tilde{\psi}\left( y\right) =\left( w\left( x,t\right) \right) ^{\frac{1}{4}%
}\psi \left( x\right) ,  \label{eqn5}
\end{equation}%
\begin{equation}
y=x-\int_{x}^{\infty }\left( \sqrt{w\left( r,t\right) }-1\right) dr,
\label{eqn6}
\end{equation}%
equation (\ref{Lax1}) can be reformulated as a eigenvalue problem of the Schr%
\"{o}dinger equation type given below%
\begin{equation}
L\tilde{\psi}:=-\tilde{\psi}_{yy}+q\left( y,t\right) \tilde{\psi}=k^{2}%
\tilde{\psi}\text{.}  \label{Lax1schrodinger}
\end{equation}%
In the above,%
\begin{equation}
\lambda =\frac{1}{4}+k^{2},  \label{eqn7}
\end{equation}%
and%
\begin{equation}
q\left( y,t\right) =\frac{w_{yy}\left( y,t\right) }{4w\left( y,t\right) }-%
\frac{3}{16}\frac{\left( w_{y}\right) ^{2}\left( y,t\right) }{w^{2}\left(
y,t\right) }+\frac{1-w\left( y,t\right) }{4w\left( y,t\right) }  \label{eqn9}
\end{equation}%
with $w\left( y,t\right) =w\left( x\left( y\right) ,t\right) .$ Given an
initial solution $u\left( x,0\right) $, the eigenvalues (discrete and
continuous spectra) and the corresponding eigenfunction of the CH equation
can be obtained by solving the solution $\tilde{\psi}\left( y,k,t\right) $
at $t=0$ from (\ref{Lax1schrodinger}):%
\begin{equation*}
-\tilde{\psi}_{yy}\left( y,k,0\right) +q\left( y,0\right) \tilde{\psi}\left(
y,k,0\right) =k^{2}\tilde{\psi}\left( y,k,0\right) \text{.}
\end{equation*}%
According to the derivation detailed in, e.g., \cite{DeiftT} or \cite%
{FaddeevS}, under the integrable condition%
\begin{equation*}
\int_{%
%TCIMACRO{\U{211d} }%
%BeginExpansion
\mathbb{R}
%EndExpansion
}\left( 1+\left \vert y\right \vert \right) ^{2}q\left( y,0\right) dy<\infty
,
\end{equation*}%
for the potential $q\left( y,0\right) $, there exists a finite number of
discrete spectra, say $k=i\mu _{1},\cdots ,i\mu _{N}$ for some $N\in
%TCIMACRO{\U{2115} }%
%BeginExpansion
\mathbb{N}
%EndExpansion
$. For each $i\mu _{j}$ with $j=1,...,N,$ its corresponding eigenfunction is
$\tilde{\psi}_{j}\left( y\right) =\tilde{\psi}\left( y,i\mu _{j},0\right) .$
Let the normalization constant be $\gamma _{j}$ appearing in the following
asymptotic expression of $\tilde{\psi}_{j}:$%
\begin{equation*}
\tilde{\psi}_{j}\left( y\right) =\gamma _{j}e^{-\mu _{j}y}+o\left( 1\right)
\text{ as }y\rightarrow \infty .
\end{equation*}%
Each $k\in
%TCIMACRO{\U{211d} }%
%BeginExpansion
\mathbb{R}
%EndExpansion
$ corresponds to a continuous spectrum. The respective eigenfunction $\hat{%
\psi}=\hat{\psi}\left( y,k\right) $ satisfies%
\begin{equation*}
\hat{\psi}\sim \left \{
\begin{tabular}{l}
$e^{-iky}+\tilde{R}\left( k\right) e^{iky}$; $\text{as }y\rightarrow \infty
,\smallskip $ \\
$\tilde{T}\left( k\right) e^{-iky}$; $\text{as }y\rightarrow -\infty .$%
\end{tabular}%
\  \  \  \right.
\end{equation*}%
where $\tilde{T}\left( k\right) $ and $\tilde{R}\left( k\right) $ are called
as the transmission and reflection coefficients, respectively. Note that $%
\left \vert \tilde{T}\left( k\right) \right \vert ^{2}+\left \vert \tilde{R}%
\left( k\right) \right \vert ^{2}=1.$

Given an initial condition $u\left( x,0\right) $, through the Liouville
transform given in (\ref{eqn5})-(\ref{eqn6}), the corresponding potential $%
q\left( y,0\right) $ and the associated scattering data $\left \{ \tilde{R}%
\left( k\right) ,\mu _{j},\gamma _{j}\right \} $ can be derived. Hence, a
map between $u\left( x,0\right) $ and the set of scattering data, namely, $%
\left \{ \tilde{R}\left( k\right) ,\mu _{j},\gamma _{j}\right \} $ exists in
the so-called direct scattering problem%
\begin{equation*}
u\left( x,0\right) \rightarrow \left \{ \tilde{R}\left( k\right) ,\mu
_{j},\gamma _{j}\right \} \text{ }(j=1,...,N).
\end{equation*}%
Analysis of the inverse scattering problem can answer whether the map
defined above can be inverted or not. The detail can be seen in \cite{AKNS},
\cite{MKST}, \cite{Constantin2001}, \cite{DeiftT} and \cite{DrazinJohnson}.
Let $T\left( k\right) $ and $R\left( k\right) $ be the respective
transmission and reflection coefficients for the original eigenvalue problem
(\ref{Lax1}), from \cite{MKST} it follows that $%
R(k,t)=R(k)e^{-ikt/(1/4+k^{2})},$ $\gamma _{j}\left( t\right) =\gamma
_{j}e^{\mu _{j}t/2(1/4-\mu _{j}^{2})},$ for $t>0.$ The relation between $%
T\left( k\right) ,R\left( k\right) $ and $\tilde{T}\left( k\right) ,\tilde{R}%
\left( k\right) $ can then be derived as%
\begin{equation*}
T\left( k\right) =\tilde{T}\left( k\right) e^{ikH_{-1}\left( w\right) },%
\text{ }R\left( k\right) =\tilde{R}\left( k\right) ,
\end{equation*}%
where $H_{-1}\left( w\right) \equiv \int_{%
%TCIMACRO{\U{211d} }%
%BeginExpansion
\mathbb{R}
%EndExpansion
}\left( \sqrt{w\left( r,t\right) }-1\right) dr$ (see \cite{MKST}). Note that
$H_{-1}\left( w\right) $ is independent of $t$ (see, for example, \cite%
{ConstantinIvanov2006}). In summary, application of the inverse scattering
approach to the initial boundary value problem for the equations (\ref{eq:CH}%
-\ref{eq:bc2copy}) involves performing a spectral analysis on the two
eigenvalue equations accounting respectively for the equations of Lax pair.

\section{Initial scattering data\label{sec: initial condition}}

A special initial data for (\ref{eq:CH}) will be derived in this section to
give us an explicit expression of the scattering data.

\begin{theorem}
\label{example}Let $q_{0}\in \left( 0,1\right) $ be a given constant and $A=%
\frac{q_{0}}{1-q_{0}}$. The solution of the Camassa-Holm equation (\ref%
{eq:CH}) is sough subject to the following initial condition%
\begin{equation}
u\left( x,0\right) =\left \{
\begin{array}{ll}
\frac{A\left( A+1+\log \left( e^{x}-A\right) \right) }{e^{x}},\medskip &
\text{for }x\geq \log \left( 1+A\right) , \\
\frac{A\left( A+1+\log \left( \left( 1+A\right) ^{2}e^{-x}-A\right) \right)
}{\left( 1+A\right) ^{2}e^{-x}}, & \text{for }x<\log \left( 1+A\right) .%
\end{array}%
\right.  \label{ini15}
\end{equation}%
For the initial solution defined above, the corresponding scattering data in
spectral domain can be derived as%
\begin{equation}
R\left( k\right) =\frac{-q_{0}}{q_{0}+2ik},\text{ }\mu _{1}=\frac{q_{0}}{2},%
\text{ }\gamma _{1}=\sqrt{\frac{q_{0}}{2}}.  \label{scatteringdata}
\end{equation}
\end{theorem}

\begin{proof}
For convenience, we write $w\left( y\right) =w\left( y,0\right) $ and $%
u\left( y\right) =u\left( y,0\right) .$ Provided that the potential takes
the form of $q\left( y,0\right) =-q_{0}\delta \left( y\right) $ with $%
q_{0}\left( >0\right) $ being specified as a given constant, the scattering
data can then be obtained in the form of (\ref{scatteringdata}) (see for
example, \cite{DrazinJohnson}). In the following, we are aimed to derive $%
w\left( y,0\right) \ $and its derivation is given below.

The following equation subject to the condition\ $\lim_{\left \vert
y\right
\vert \rightarrow \infty }C\left( y\right) =1$ is considered%
\begin{equation}
C_{yy}=C\left( q\left( y,0\right) +\frac{1}{4}\right) -\frac{1}{4C^{3}}.
\label{C eq'n t=0}
\end{equation}%
From \cite{Constantin2001} it follows that $w\left( y,0\right) =C^{4}\left(
y\right) $ is the solution of (\ref{eqn9}). Consider (\ref{C eq'n t=0}) in $%
y>0.$ Multiply $C_{y}$ on both sides of $C_{yy}=\frac{C}{4}-\frac{1}{4C^{3}}$
and then integrate the resulting equation from $y$ to $\infty ,$ yielding%
\begin{equation*}
\left. \frac{\left( C_{y}\right) ^{2}}{2}\right \vert _{y}^{\infty
}=\int_{y}^{\infty }\left( \frac{C\left( z\right) }{4}-\frac{1}{4C^{3}\left(
z\right) }\right) C_{z}\left( z\right) dz=\int_{C\left( y\right) }^{1}\left(
\frac{C}{4}-\frac{1}{4C^{3}\left( z\right) }\right) dC.
\end{equation*}%
Then we have $C_{y}=\pm \frac{1}{2}\left( C-\frac{1}{C}\right) .$ The
negative sign considered above is based on the fact that $C\left( y\right)
\rightarrow 1$ as $y\rightarrow \infty $. After integration, we can obtain $%
C\left( y\right) =\left( \left( C_{0}^{2}-1\right) e^{-y}+1\right) ^{\frac{1%
}{2}}$ for $y>0,$ where $C_{0}$ is an integration constant. For $y<0,$ the
derivation procedure is similar to the case for $y>0.$ Therefore $C\left(
y\right) $ can be expressed as%
\begin{equation}
C\left( y\right) =\left( \left( C_{0}^{2}-1\right) e^{-\left \vert y\right
\vert }+1\right) ^{\frac{1}{2}}\text{.}  \label{C formula}
\end{equation}%
Given $\varepsilon >0$, integration of the equation (\ref{C eq'n t=0}) from $%
-\varepsilon $ to $\varepsilon $ yields%
\begin{equation*}
C_{y}\left( \varepsilon \right) -C_{y}\left( -\varepsilon \right)
=\int_{-\varepsilon }^{\varepsilon }\left[ C\left( y\right) \left(
-q_{0}\delta \left( y\right) +\frac{1}{4}\right) -\frac{1}{4C^{3}\left(
y\right) }\right] dy.
\end{equation*}%
By letting $\varepsilon \rightarrow 0,$ we can have $C_{y}\left(
0^{+}\right) -C_{y}\left( 0^{-}\right) =-q_{0}C\left( 0\right) .$
Substituting (\ref{C formula}) into the above equation, it follows that $%
C_{0}^{2}=\frac{1}{1-q_{0}}.$ The functional form for $w\left( y,0\right) $
shown below\ can then be derived provided that $q_{0}\in \left( 0,1\right) ,$%
\begin{equation}
w\left( y,0\right) =C^{4}\left( y\right) =\left( Ae^{-\left \vert y\right
\vert }+1\right) ^{2}.  \label{w(y,0))}
\end{equation}

In the next step, the relation between $y$ and $x$ in (\ref{eqn6}) will be
derived. From (\ref{eqn6}), $y$ and $x$ satisfy%
\begin{equation}
\frac{dy}{dx}=\left( w\left( y,0\right) \right) ^{\frac{1}{2}},\text{ }x\in
%TCIMACRO{\U{211d} }%
%BeginExpansion
\mathbb{R}
%EndExpansion
\label{y to x eq'n}
\end{equation}%
under the boundary condition $\lim_{x\rightarrow \infty }\left( y\left(
x\right) -x\right) =0.$ Substituting (\ref{w(y,0))}) into (\ref{y to x eq'n}%
), we obtain%
\begin{equation}
\frac{dy}{dx}=1+Ae^{-\left \vert y\right \vert }.  \label{dy/dx}
\end{equation}%
Integration of (\ref{dy/dx}) leads to $y=\log \left( e^{x+c_{1}}-A\right) $
for $y>0,$ where $c_{1}$ is an integration constant. By employing the
condition $\lim_{x\rightarrow \infty }\left( y\left( x\right) -x\right) =0,$
$c_{1}=0$ can be obtained. Then, we can have $y=\log \left( e^{x}-A\right) $
for $y>0$, that is, $x>\log \left( 1+A\right) .$

Integration of equation (\ref{dy/dx}) in $y<0$ yields $y=\log \frac{%
e^{x+c_{2}}}{1-Ae^{x+c_{2}}},$ where $c_{2}$ is an integration constant.
Here $y<0$ is equivalent to $x<-\log \left( 1+A\right) -c_{2}.$

Let $y=y\left( x\right) $ be defined on the whole $x$-axis. By choosing $%
c_{2}$ such that $\log \left( 1+A\right) =-\log \left( 1+A\right) -c_{2}$,
thereby yielding $e^{c_{2}}=\frac{1}{\left( 1+A\right) ^{2}}.$ Then, we can
get%
\begin{equation}
y=\left \{
\begin{array}{ll}
\log \left( e^{x}-A\right) ,\medskip & \text{for }x\geq \log \left(
1+A\right) , \\
-\log \left( \left( 1+A\right) ^{2}e^{-x}-A\right) , & \text{for }x<\log
\left( 1+A\right) .%
\end{array}%
\right.  \label{y(x)}
\end{equation}%
From (\ref{w(y,0))}), it can be observed that $w_{y}\left( y,0\right) =-$sgn$%
\left( y\right) 2Ae^{-\left \vert y\right \vert }\left( Ae^{-\left \vert
y\right \vert }+1\right) .$ Recall $w=u-u_{xx}+1$ if $u$ and $w$ are
functions of $x$. In $y$ domain, by (\ref{y to x eq'n}), $w=u-u_{xx}+1$
turns out to be equivalent to $wu_{yy}+\frac{1}{2}w_{y}u_{y}-u=1-w.$ From (%
\ref{w(y,0))}), this equation can be written as%
\begin{equation}
u_{yy}-\text{sgn}\left( y\right) \frac{A}{A+e^{\left \vert y\right \vert }}%
u_{y}-\frac{1}{\left( Ae^{-\left \vert y\right \vert }+1\right) ^{2}}u=\frac{%
1}{\left( Ae^{-\left \vert y\right \vert }+1\right) ^{2}}-1.
\label{u eq'n in y domain>0}
\end{equation}%
We solve firstly the solution $u$ in the half domain $y>0$. The homogeneous
equation of (\ref{u eq'n in y domain>0}) can be shown to have a fundamental
set of solutions, namely%
\begin{equation*}
u_{1}\left( y\right) :=\frac{-A\left( \frac{1}{2}e^{y}+A\right) }{Ae^{-y}+1},%
\text{ }u_{2}\left( y\right) :=\frac{-Ae^{-y}}{Ae^{-y}+1}.
\end{equation*}%
We suppose that the solution of (\ref{u eq'n in y domain>0}) is $u\left(
y\right) =c_{3}\left( y\right) u_{1}\left( y\right) +c_{4}\left( y\right)
u_{2}\left( y\right) $ for some $c_{3}\left( y\right) $ and $c_{4}\left(
y\right) $ to be determined. Substitution of this expression into (\ref{u
eq'n in y domain>0}) yields%
\begin{equation*}
c_{3}^{\prime }\left( y\right) =\frac{-1}{A}e^{-y}\left( \frac{1}{\left(
Ae^{-y}+1\right) ^{2}}-1\right) ,\text{ \ }c_{4}^{\prime }\left( y\right) =%
\frac{1}{A}\left( \frac{1}{2}e^{y}+A\right) \left( \frac{1}{\left(
Ae^{-y}+1\right) ^{2}}-1\right) .
\end{equation*}%
Integrating the above two expressions for $y>0,$ we have%
\begin{align*}
u\left( y\right) & =\frac{-A\left( \frac{1}{2}e^{y}+A\right) }{Ae^{-y}+1}%
\left( \frac{-Ae^{-2y}}{Ae^{-y}+1}+c_{5}\right) -\frac{Ae^{-y}}{Ae^{-y}+1}%
\left( \frac{-1}{2}\left( \frac{2Aye^{-y}-Ae^{-y}+2y}{Ae^{-y}+1}\right)
+c_{6}\right) \\
& =\frac{A\left( 2y+1\right) e^{-y}}{2\left( Ae^{-y}+1\right) }-c_{5}\frac{%
A\left( e^{y}+2A\right) }{2\left( Ae^{-y}+1\right) }-c_{6}\frac{Ae^{-y}}{%
Ae^{-y}+1}.
\end{align*}%
Since we want to get the bounded solution for $y\in
%TCIMACRO{\U{211d} }%
%BeginExpansion
\mathbb{R}
%EndExpansion
,$ $c_{5}\equiv 0$ turns out to be the consequence. Owing to the symmetry of
the solution $u\left( y\right) $ about $y=0,$ we have%
\begin{equation*}
u\left( y\right) =\frac{A\left( 2\left \vert y\right \vert +1\right)
e^{-\left \vert y\right \vert }}{2\left( Ae^{-\left \vert y\right \vert
}+1\right) }+c_{7}\frac{Ae^{-\left \vert y\right \vert }}{Ae^{-\left \vert
y\right \vert }+1}=\frac{Ae^{-\left \vert y\right \vert }}{2\left(
Ae^{-\left \vert y\right \vert }+1\right) }\left( 2\left \vert y\right \vert
+1+c_{8}\right) .
\end{equation*}%
In order for $u\in C^{1}\left(
%TCIMACRO{\U{211d} }%
%BeginExpansion
\mathbb{R}
%EndExpansion
\right) ,$ we need to choose $c_{8}=2A+1,$ thereby leading to $u\left(
y,0\right) =\frac{A\left( \left \vert y\right \vert +A+1\right) }{\left(
A+e^{\left \vert y\right \vert }\right) }.$ By substituting the expression (%
\ref{y(x)}) into the above equation, we can derive the initial solution of
the form given in (\ref{ini15}). The proof is completed.
\end{proof}

\noindent \textbf{Remark}. From (\ref{w(y,0))}) and (\ref{y(x)}), we can
also get the expression of $w\left( x,0\right) :$%
\begin{equation*}
w\left( x,0\right) =\left \{
\begin{array}{ll}
\left( \frac{1}{1-Ae^{-x}}\right) ^{2},\medskip & \text{for }x\geq \log
\left( 1+A\right) , \\
\left( \frac{\left( 1+A\right) ^{2}}{\left( 1+A\right) ^{2}-Ae^{x}}\right)
^{2}, & \text{for }x<\log \left( 1+A\right) .%
\end{array}%
\right.
\end{equation*}%
Then, $w\left( x,0\right) $ satisfies the requirement (\ref{eq:bc2copy}).
The need of solving the inverse scattering problem \cite{MKST,Constantin2001}
becomes clear. By \cite{Constantin2001}, we have $w\left( x,t\right) >0$ for
all $x\in
%TCIMACRO{\U{211d} }%
%BeginExpansion
\mathbb{R}
%EndExpansion
$ and for all $t>0.$

According to the work of \cite{MKST}, the long time asymptotic behavior of
the CH solution depends largely on the number of discrete spectra $-\mu
_{1}^{2},\cdots ,-\mu _{N}^{2}$ and the existence of the reflection
coefficient $\tilde{R}\left( k\right) $ of (\ref{Lax1schrodinger}). Thanks
to Theorem \ref{example}, we know that only one discrete spectrum $-\mu
_{1}^{2}$ ($N=1$) exists. Therefore, only one soliton is permitted to appear
in the "soliton region" \cite{MKST}. This will be revealed in the numerical
results shown in section 5.

\section{Symplecticity and dispersion relation equation preserving scheme in
space-time domain \label{sec:Scheme}}

Camassa-Holm equation has many remarkable structure preserving features. All
of them should be taken into account while developing a reliable finite
difference scheme in time domain for getting a physically correct CH
solution. In what follows, some of the mathematical properties that are
directly related to our scheme development are recalled. First of all, there
exist two compatible Hamiltonian descriptions of the Camassa-Holm equation.
The first one is expressed as $m_{t}=-D_{1}\frac{\delta H_{1}}{\delta m},$
where $D_{1}=m\frac{\partial }{\partial x}+\frac{\partial }{\partial x}m$
and $H_{1}=\frac{1}{2}\int u^{2}+(u_{x})^{2}~dx$. The second one is $%
m_{t}=-D_{2}\frac{\delta H_{2}}{\delta m},$ where $D_{2}=\frac{\partial }{%
\partial x}+\frac{\partial ^{3}}{\partial x^{3}}$ and $H_{2}=\frac{1}{2}\int
u^{3}+u(u_{xx})^{2}~dx$. A scheme applicable to solve equations (\ref{eq:CH}-%
\ref{eq:bc2copy}) should discretely conserve Hamiltonians cast in the
corresponding integrable forms. A symplecticity-preserving temporal scheme
developed in Section 3.2 will be applied to the current approximation of the
time derivative term. The nonlinear CH equation is highly dispersive. It is
therefore important to develop a dispersion-error reducing scheme in Section
3.1 to accurately approximate the spatial derivative terms in CH equation.

%%%%%%%%%%%%%%%%%%%%%%%%%%%%%%%%%%%%%%%%%%%%%%%%%%%%%%%%%%%%%%%%%%%%%%%%%%%%%%
\  \  \  \ To avoid approximating the space-time mixed derivative and the
third-order dispersive terms, equation (\ref{eq:CH}) is rewritten to its
equivalent system of equations possessing only the first-order derivative
terms. Subject to a properly prescribed boundary condition and an initial
condition $u(x,t=0)=f\in H^{1}$, the solution to (\ref{eq:CH}) is sought
from the following inhomogeneous hyperbolic nonlinear pure advection
equation for $u$%
\begin{equation}
u_{t}+uu_{x}=-P_{x}.  \label{eq:one-step1}
\end{equation}%
The pressure-like variable $P$ shown above in the right hand side is
governed by the following elliptic Helmholtz equation rather than by the
Poisson equation encountered in the inviscid Euler or the Navier-Stokes
fluid flow%
\begin{equation}
P-P_{xx}=u^{2}+uu_{x}.  \label{eq:one-step-p}
\end{equation}%
In the light of the above two equations, equation (\ref{eq:CH}) is without
doubt classified to be elliptic-hyperbolic provided that the solution
remains smooth.

%%%%%%%%%%%%%%%%%%%%%%%%%%%%%%%%%%%%%%%%%%%%%%%%%%%%%%%%%%%%%%%%%%%%%%%%%%%%%%%%%%%%%%%

\subsection{Approximation of spatial derivatives}

Within the framework of the combined compact difference (CCD) schemes, the
first derivative term $\frac{\partial u}{\partial x}$ and the second
derivative term $\frac{\partial ^{2}u}{\partial x^{2}}$ are approximated
implicitly in a grid of three-point stencil as follows for the case of $u>0$%
\begin{align}
& a_{1}\frac{\partial u}{\partial x}|_{i-1}+\frac{\partial u}{\partial x}%
|_{i}+a_{3}\frac{\partial u}{\partial x}|_{i+1}  \notag  \label{eq:CCD-EQ1}
\\
=\frac{1}{h}(c_{1}u_{i-2}+c_{2}& u_{i-1}+c_{3}u_{i})-h\left( b_{1}\frac{%
\partial ^{2}u}{\partial x^{2}}|_{i-1}+b_{2}\frac{\partial ^{2}u}{\partial
x^{2}}|_{i}+b_{3}\frac{\partial ^{2}u}{\partial x^{2}}|_{i+1}\right) , \\
-\frac{1}{8}\frac{\partial ^{2}u}{\partial x^{2}}|_{i-1}+\frac{\partial ^{2}u%
}{\partial x^{2}}|_{i}-\frac{1}{8}\frac{\partial ^{2}u}{\partial x^{2}}%
|_{i+1}& =\frac{3}{h^{2}}(u_{i-1}-2u_{i}+u_{i+1})-\frac{9}{8h}\left( -\frac{%
\partial u}{\partial x}|_{i-1}+\frac{\partial u}{\partial x}|_{i+1}\right) .
\label{eq:CCD-EQ2}
\end{align}%
The coefficients shown in (\ref{eq:CCD-EQ2}) have been determined solely
from the modified equation analysis, yielding a formal accuracy order of six
\cite{bib:Chu98}. The coefficients in (\ref{eq:CCD-EQ1}) are then derived by
performing the Taylor series expansion on the terms $u_{i-1}$, $u_{i+1}$, $%
\frac{\partial u}{\partial x}|_{i-1}$, $\frac{\partial u}{\partial x}|_{i}$,
$\frac{\partial u}{\partial x}|_{i+1}$, $\frac{\partial ^{2}u}{\partial x^{2}%
}|_{i-1}$, $\frac{\partial ^{2}u}{\partial x^{2}}|_{i}$ and $\frac{\partial
^{2}u}{\partial x^{2}}|_{i+1}$ with respect to $u_{i}$ from the derived
modified equation for (\ref{eq:CCD-EQ1}). One more algebraic equation
derived below is needed to uniquely determine all the introduced
coefficients shown in (\ref{eq:CCD-EQ1}).

The strategy of deriving the last algebraic equation is to reduce the
dispersion error by matching the exact and numerical wavenumbers. Performing
an error reduction procedure amounts to equating the effective wavenumbers $%
\alpha ^{^{\prime }}$ and $\alpha ^{^{\prime \prime }}$ to those shown in
the right-hand sides of the following equations. (\ref{eq:transform3}) and (%
\ref{eq:transform4}) \cite{bib:Tam93}.%
\begin{equation}
\mathbf{i}\alpha ^{^{\prime }}h~(a_{1}e^{-\mathbf{i}\alpha h}+1+a_{3}e^{%
\mathbf{i}\alpha h})=(c_{1}e^{-2\mathbf{i}\alpha h}+c_{2}e^{-\mathbf{i}%
\alpha h}+c_{3})-(\mathbf{i}\alpha ^{^{\prime \prime }}h)^{2}(b_{1}e^{-%
\mathbf{i}\alpha h}+b_{2}+b_{3}e^{\mathbf{i}\alpha h}),
\label{eq:transform3}
\end{equation}%
\begin{equation}
(\mathbf{i}\alpha ^{^{\prime \prime }}h)^{2}(-\frac{1}{8}e^{-\mathbf{i}%
\alpha h}+1-\frac{1}{8}e^{\mathbf{i}\alpha h})=(3e^{-\mathbf{i}\alpha
h}-6+3e^{\mathbf{i}\alpha h})-\mathbf{i}\alpha ^{^{\prime }}h~(-\frac{9}{8}%
e^{-\mathbf{i}\alpha h}+\frac{9}{8}e^{\mathbf{i}\alpha h}).
\label{eq:transform4}
\end{equation}%
The expression of $\alpha ^{\prime }h$ can then be derived from the above
two equations. The real and imaginary parts of the numerical modified (or
scaled) wavenumber $\alpha ^{\prime }h$ account respectively for the
dispersion error (or phase error) and the dissipation error (or amplitude
error).

To get a higher dispersive accuracy for $\alpha ^{\prime }$, we demand that
the value of $\alpha h$ should be sufficiently closer to the real part of $%
\alpha ^{\prime }h$ or $\Re \lbrack \alpha ^{\prime }h]$. The error function
$E(\alpha )$ defined below should be a positive minimum over the following
integration interval of the modified wavenumber $\alpha h$
\begin{equation}
E(\alpha )=\int_{0}^{\frac{7\pi }{8}}\left[ W\left( \alpha h-\Re \lbrack
\alpha ^{\prime }h]\right) \right] ^{2}d(\alpha h).  \label{eq:error_DRP}
\end{equation}%
The weighting function $W$ in (\ref{eq:error_DRP}) is chosen to be the
denominator of $\left( \alpha h-\Re \lbrack \alpha ^{\prime }h]\right) $ so
that we can integrate $E(\alpha )$ analytically \cite{bib:Ashcroft2003}. To
make the error function defined in the modified wavenumber range of $0\leq
\alpha h\leq \frac{7\pi }{8}$ to be positive or minimal, the extreme
condition given by $\frac{\partial E}{\partial c_{3}}=0$ is enforced. This
constraint equation for the sake of maximizing dispersion accuracy is used
together with the other seven algebraic equations derived previously through
the modified equation analysis. Employment of the two rigorous analysis
means described above enables us to get $a_{1}=0.888251792581$, $%
a_{3}=0.049229651564$, $b_{1}=0.150072398996$, $b_{2}=-0.250712794122$, $%
b_{3}=-0.012416467490$, $c_{1}=0.016661718438$, $c_{2}=-1.970804881023$ and $%
c_{3}=1.954143162584$. The resulting upwinding difference scheme developed
on theoretical basis in a grid stencil of three points $i-1$, $i$ and $i+1$
for $\frac{\partial u}{\partial x}$ has the spatial accuracy of order six
according to the derived modified equation, namely, $\frac{\partial u}{%
\partial x}=\frac{\partial u}{\partial x}|_{exact}+0.424003657\times
10^{-6}~h^{6}:\frac{\partial ^{7}u}{\partial x^{7}}+H.O.T.$ For $u<0$, the
proposed three-point non-centered CCD scheme can be similarly derived.
%for the derivative term $\frac{\partial u}{\partial x}$
%\begin{align}
%\label{eq:CCD-EQ1-2}\nonumber
% 0.049229651564\frac{\partial u}{\partial x}|_{i-1}
%  +\frac{\partial u}{\partial x}|_{i}
%  +0.888251792581\frac{\partial u}{\partial x}|_{i+1}\\\nonumber
%+ h ( 0.012416467490\frac{\partial^2 u}{\partial x^2}|_{i-1}
%          +0.250712794122\frac{\partial^2 u}{\partial x^2}|_{i}
%          -0.150072398996\frac{\partial^2 u}{\partial x^2}|_{i+1} ) \\
%  =\frac{1}{h}( -1.954143162584 u_{i} + 1.970804881023u_{i+1} - 0.016661718438u_{i+2} ).
%\end{align}

%The above derived sixth-order accurate upwinding combined compact difference scheme
%is also applied to approximate the gradient term $P_x$ shown in (\ref{eq:one-step1}).
The three-point combined compact difference (CCD) scheme \cite{bib:Chu98} is
used here to approximate the gradient terms $P_{x}$ shown in (\ref%
{eq:one-step1}) as follows
\begin{align}
\frac{h}{16}\frac{\partial ^{2}P}{\partial x^{2}}|_{i-1}-\frac{h}{16}\frac{%
\partial ^{2}P}{\partial x^{2}}|_{i+1}& =\frac{15}{16h}(-P_{i-1}+P_{i+1})+%
\left( \frac{7}{16}\frac{\partial P}{\partial x}|_{i-1}+\frac{\partial P}{%
\partial x}|_{i}+\frac{7}{16}\frac{\partial P}{\partial x}|_{i+1}\right) ,
\label{eq:CCD-EQ1-wenni} \\
-\frac{1}{8}\frac{\partial ^{2}P}{\partial x^{2}}|_{i-1}+\frac{\partial ^{2}P%
}{\partial x^{2}}|_{i}-\frac{1}{8}\frac{\partial ^{2}P}{\partial x^{2}}%
|_{i+1}& =\frac{1}{h^{2}}(3P_{i-1}-6P_{i}+3P_{i+1})-\frac{1}{h}\left( -\frac{%
9}{8}\frac{\partial P}{\partial x}|_{i-1}+\frac{9}{8}\frac{\partial P}{%
\partial x}|_{i+1}\right) .  \label{eq:CCD-EQ2-wenni}
\end{align}%
The above centered CCD scheme developed on theoretical basis in a stencil of
three grid points $i-1$, $i$ and $i+1$ for $\frac{\partial P}{\partial x}$
has the sixth-order accuracy.

The Helmholtz equation (or equation (\ref{eq:one-step-p})) for $P_{i}$ is
approximated as follows for
%$g_{i} = -(u^{2}_{i}+ u_{i}u_{x,i}+2\kappa u_{i})  $
$g_{i}=-(u_{i}^{2}+u_{i}u_{x,i})$
\begin{align}
& P_{i+1}-\left( 2+h^{2}+\frac{1}{12}h^{4}+\frac{1}{360}h^{6}\right)
P_{i}+P_{i-1}  \notag  \label{eq:eq16} \\
& =h^{2}g_{i}+\frac{1}{12}h^{4}\left( f_{i}+\frac{\partial ^{2}g_{i}}{%
\partial x^{2}}\right) +\frac{1}{360}h^{6}\left( g_{i}+\frac{\partial
^{2}g_{i}}{\partial x^{2}}+\frac{\partial ^{4}g_{i}}{\partial x^{4}}\right) .
\end{align}%
The corresponding modified equation for equation (\ref{eq:one-step-p}) is $%
\frac{\partial ^{2}P}{\partial x^{2}}-P=g+\frac{h^{6}}{20160}\frac{\partial
^{8}P}{\partial x^{8}}+\frac{h^{8}}{1814400}\frac{\partial ^{10}P}{\partial
x^{10}}+\cdots +H.O.T.$. This proposed three-point compact difference scheme
is therefore sixth-order accurate.

%%%%%%%%%%%%%%%%%%%%%%%%%%%%%%%%%%%%%%%%%%%%%%%%%%%%%%%%%%%%%%%%%%%%%%%%%%%%%%%%%%%%%%%

\subsection{Approximation of temporal derivatives}

%%%%%%%%%%%%%%%%%%%%%%%%%%%%%%%%%%%%%%%%%%%%%%%%%%%%%%%%%%%%%%%%%%%%%%%%%%%%%%%%%%%%%%%
\  \  \  \ Since equation (\ref{eq:one-step1}) has a symplectic structure, the
time-stepping scheme cannot be chosen arbitrarily provided that a long-term
accurate solution is sought. To conserve symplectic property existing in the
currently investigated non-dissipative Hamiltonian system of equations (\ref%
{eq:one-step1}-\ref{eq:one-step-p}), the sixth-order accurate symplectic
Runge-Kutta scheme \cite{bib:Oevel(1997)} given below is applied iteratively
to accurately integrate the CH equation
\begin{align}
u^{(1)}=u^{n}& +\Delta t\left[ \frac{5}{36}F^{(1)}+(\frac{2}{9}+\frac{2%
\tilde{c}}{3})F^{(2)}+(\frac{5}{36}+\frac{\tilde{c}}{3})F^{(3)}\right] ,
\label{eq:RK1} \\
u^{(2)}=u^{n}& +\Delta t\left[ (\frac{5}{36}-\frac{5\tilde{c}}{12})F^{(1)}+(%
\frac{2}{9})F^{(2)}+(\frac{5}{36}+\frac{5\tilde{c}}{12})F^{(3)}\right] ,
\label{eq:RK2} \\
u^{(3)}=u^{n}& +\Delta t\left[ (\frac{5}{36}-\frac{\tilde{c}}{3})F^{(1)}+(%
\frac{2}{9}-\frac{2\tilde{c}}{3})F^{(2)}+\frac{5}{36}F^{(3)}\right] ,
\label{eq:RK3} \\
u^{n+1}=u^{n}& +\Delta t\left[ \frac{5}{18}F^{(1)}+\frac{4}{9}F^{(2)}+\frac{5%
}{18}F^{(3)}\right] .  \label{eq:RK4}
\end{align}%
In the above, $\tilde{c}=\frac{1}{2}\sqrt{\frac{3}{5}}$ and $%
F^{(i)}=F(u^{(i)},P^{(i)})$, $i=1,2,3$.

%The solution $u^{n+1}$ calculated from equation (\ref{eq:RK4})
%needs to solve equations (\ref{eq:RK1}) $-$ (\ref{eq:RK3}) iteratively
%for obtaining the values of $u^{(1)}$, $u^{(2)}$ and $u^{(3)}$.
%The Helmholtz equation (\ref{eq:one-step-p}) is then solved
%to get $P^{(1)}$, $P^{(2)}$ and $P^{(3)}$.
%Upon reaching the specified convergence criteria,
%we can get first the solution $u^{n+1}$ and then the solution $P^{n+1}$.
%This iterative procedure is repeated until the difference
%cast in $L_{2}$-norm form
%of the solutions calculated from two consecutive steps
%falls below $10^{-9}$.

%%%%%%%%%%%%%%%%%%%%%%%%%%%%%%%%%%%%%%%%%%%%%%%%%%%%%%%%%%%%%%%%%%%%%%%%%%%%%%

\section{Long-time asymptotics of the CH equation\label{long time}}

In \cite{MKST}, CH equation (\ref{eq:CH}) has been transformed to its
corresponding Riemann-Hilbert problem. The long time asymptotics can then be
derived by using the nonlinear steepest descent method \cite{DeiftZhou} and
the isomonodromy method \cite{FIKN}. Four solution sectors and two
transition regions have been found in the $\left( x,t>0\right) $ half-plane.
The leading term of the long-time asymptotic solution $u\left( x,t\right) $
in each region behaves differently, depending on the slope, namely, $\zeta =%
\frac{x}{t}$ (see Figure \ref{fourregionMKST}).

$(i)$ $\zeta >2+\varepsilon $ for any $\varepsilon >0:$ soliton region,

$(ii)$ $0\leq \zeta <2-\varepsilon $ for any $\varepsilon >0:$ first
oscillatory\ region,

$(iii)$ $\frac{-1}{4}+\varepsilon <\zeta <0$ for any $\varepsilon >0:$
second oscillatory\ region,

$(iv)$ $\zeta <\frac{-1}{4}-\varepsilon $ for any $\varepsilon >0:$ fast
decay\ region.

\noindent Moreover, by \cite{bib:Monvel(2010)}, the following two transition
regions exist for any $C>0:$

(1) $\left \vert \zeta -2\right \vert t^{\frac{2}{3}}<C$ (between the
solution sectors $(i)$ and $(ii)$ shown in Figure \ref{fourregionMKST}),

(2) $\left \vert \zeta +\frac{1}{4}\right \vert t^{\frac{2}{3}}<C$ (between
the solution sectors $(iii)$ and $(iv)$ shown in Figure \ref{fourregionMKST}%
).

\begin{figure}[h]
\includegraphics[width=.8\textwidth]{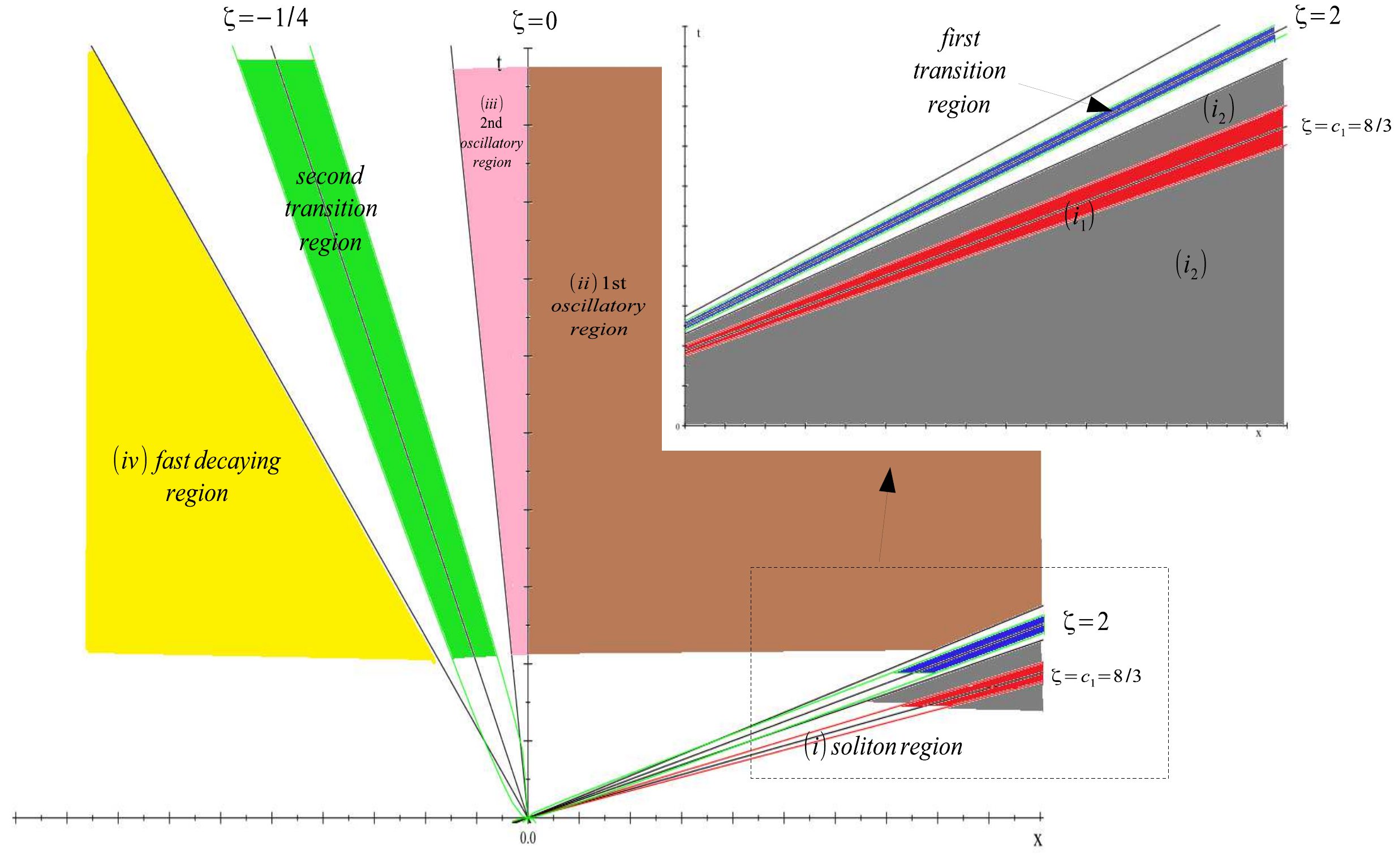}
\caption{{\protect \small The long-time asymptotics of the CH solutions in
the qualitatively different regions with }$q_{0}=\frac{1}{2}${\protect \small %
. Then }$c_{1}${\protect \small \ considered in the region (i) in Section 5.1
can be found as }$c_{1}=\frac{8}{3}.${\protect \small \ Note that }$\protect%
\zeta ${\protect \small \ is defined as }$\protect \zeta =\frac{x}{t}.$%
{\protect \small \ For }$\protect \varepsilon =\frac{14}{80},${\protect \small %
\ the soliton region }$(i)${\protect \small \ (}$\frac{x}{t}>2+\protect%
\varepsilon ${\protect \small ) which is the union of the red region }$%
(i_{1}) ${\protect \small \ (between the lies }$\frac{x}{t}=\frac{8}{3}-%
\protect \varepsilon ${\protect \small \ and }$\frac{x}{t}=\frac{8}{3}+\protect%
\varepsilon ${\protect \small ) and the gray region }$(i_{2})${\protect \small %
\ (two sectors, one is between the lines }$\frac{x}{t}=2+\protect \varepsilon
${\protect \small \ and }$\frac{x}{t}=\frac{8}{3}-\protect \varepsilon ,$%
{\protect \small \ the other one is between the lines }$\frac{x}{t}=\frac{8}{3%
}+\protect \varepsilon ${\protect \small \ and }$t=0${\protect \small ).}}
\label{fourregionMKST}
\end{figure}

Subject to the newly derived initial condition (\ref{ini15}), the long-time
asymptotic solutions in the six regions can then be expressed in terms of
the scattering data (\ref{scatteringdata}) or equivalently in terms of the
given parameter $q_{0}\in \left( 0,1\right) $ appearing in (\ref{ini15}).
The details can be seen in the Appendix A. In the following, the solution of
(\ref{eq:CH}) is sought subject to (\ref{ini15}) by applying the numerical
scheme detailed in Section \ref{sec:Scheme}. The FDTD solution $u_{\text{num}%
}\left( x,t\right) $ will be compared with the asymptotic form of the
solution expressed in terms of the scattering data (\ref{scatteringdata})
given in \cite{MKST}.

In Figure \ref{fourregionYU} the profile of the solution $u_{\text{num}%
}\left( x,t\right) $ predicted, for example, at $t=40$ is plotted with $%
q_{0}=\frac{1}{2}$. In the rest of this paper we will numerically revisit
these asymptotics summarized above in the physically distinct regions with $%
q_{0}=\frac{1}{2}$.

\begin{figure}[h]
\includegraphics[width=.99\textwidth]{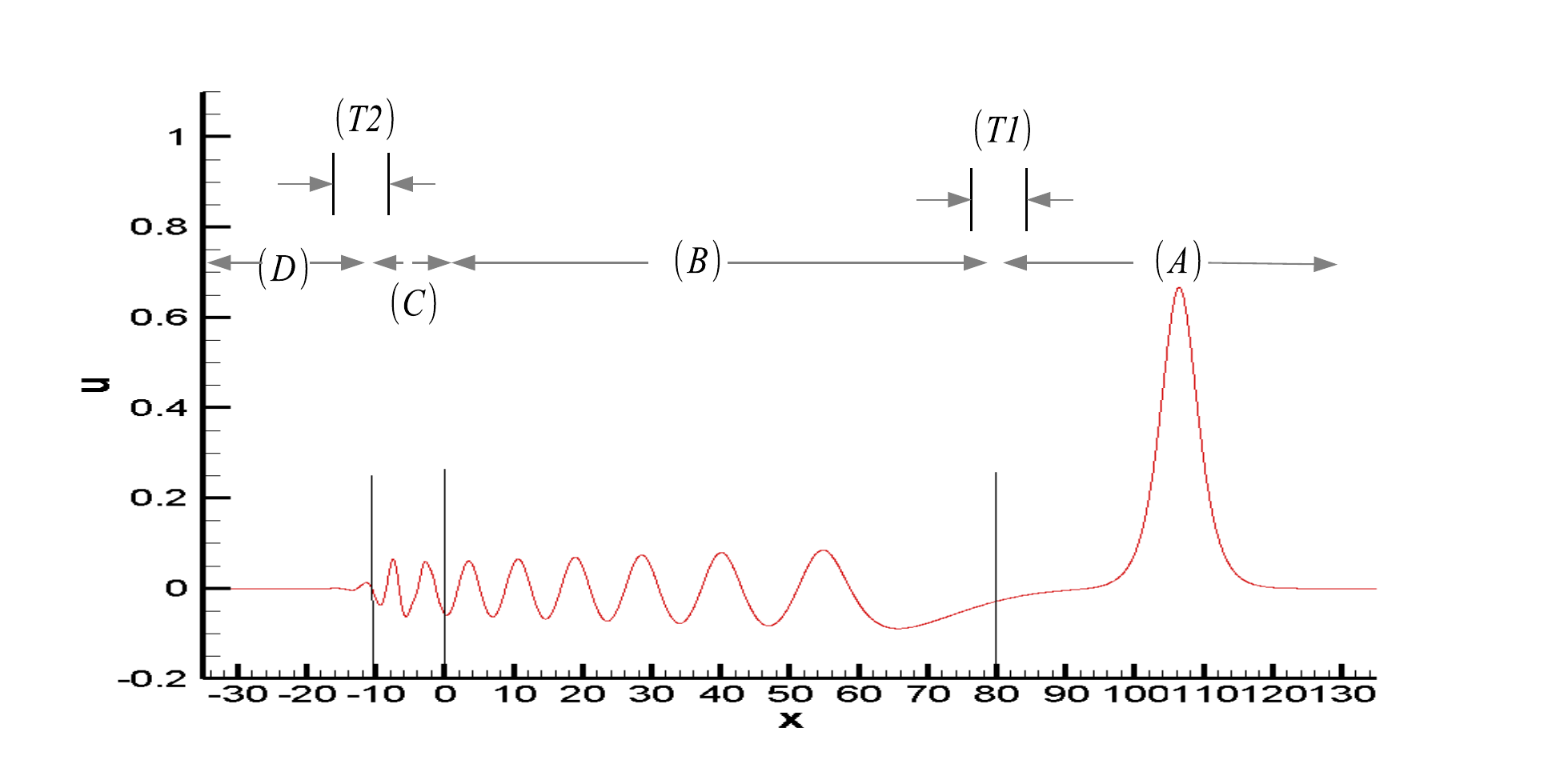}
\caption{{\protect \small The computed long-time asymptotics of the FDTD
solutions for the CH equation in four different solution regions and in two
transition regions at }$t=40${\protect \small . Note that (A) soliton region;
(B) slowly decaying modulated oscillation region; (C) region of the sum of
two decaying modulated oscillations; (D) fast decaying region; (T1) first
transition region; (T2) second transition region.}}
\label{fourregionYU}
\end{figure}

\subsection{Discussion of results in four regions $(i)$ - $(iv)$}

\subsubsection{Solution in the soliton region $(i)$}

In the soliton region $\zeta >2+\varepsilon ,$ $u\left( x,t\right) $ in $(i)$
behaves like a sole soliton due to the existence of a single discrete
eigenvalue. Let $c_{1}=\frac{2}{1-4\mu _{1}^{2}}$. For any $\varepsilon >0,$
the region $(i)$ can be divided into the following two subdomains (see \cite%
{MKST} or the Appendix A):

$(i_{1})$ If $\left \vert \zeta -c_{1}\right \vert <\varepsilon $,%
\begin{equation}
u\left( x,t\right) =u_{\text{sol}}^{1}\left( x,t\right) +O\left(
t^{-l}\right) \text{ for any }l>0,  \label{soliton region error}
\end{equation}%
where $u_{\text{sol}}^{1}\left( x,t\right) $ is the 1-soliton solution and
it can be expressed parametrically as (\ref{soliton}) and (\ref%
{1-solitonFOUMULA1}) in the Appendix A with $j=1$.

$(i_{2})$ If $\left \vert \zeta -c_{1}\right \vert \geq \varepsilon $, $%
u\left( x,t\right) $ is rapidly decreasing.$\smallskip $

According to (\ref{soliton wrt q}) in the Appendix A, the dependence of the
1-soliton solution on the parameter $q_{0}$ can be clearly seen. Since $%
q_{0}=\frac{1}{2}$, it can be found that $c_{1}=\frac{8}{3},$ and (\ref%
{soliton wrt q}) takes the form as:%
\begin{equation}
\left \{
\begin{tabular}{l}
$u\left( y,t\right) =\frac{4}{3}\frac{1}{\frac{5}{4}+\frac{3}{4}\left( e^{%
\frac{1}{2}\left( y-\frac{8}{3}t\right) }+\frac{1}{4}e^{\frac{-1}{2}\left( y-%
\frac{8}{3}t\right) }\right) },\medskip $ \\
$x\left( y,t\right) =y+\log \frac{1+\frac{3}{2}e^{-\frac{1}{2}\left( y-\frac{%
8}{3}t\right) }}{1+\frac{1}{6}e^{-\frac{1}{2}\left( y-\frac{8}{3}t\right) }}%
. $%
\end{tabular}%
\right.  \label{q=1/2 1soliton}
\end{equation}

Equation (\ref{q=1/2 1soliton}) is valid in the region ($i_{1}$): $\left
\vert \frac{x}{t}-\frac{8}{3}\right \vert <\varepsilon $. Take $\varepsilon =%
\frac{14}{80}$ as an example, at $t=80,$ the soliton region (A) shown in
Figure \ref{fourregionYU} is known to fall in the region $(i_{1})$, namely, $%
199.\, \allowbreak 33<x<227.\, \allowbreak 33$. Outside this region, i.e., $%
x\geq 227.6633$ or $174<x\leq 199.33,$ $u\left( x,t\right) \sim 0.$ The
numerically predicted solution $u_{\text{num}}\left( x,t\right) $ and the
asymptotic solution $u_{\text{sol}}^{1}\left( x,80\right) $ in (\ref{q=1/2
1soliton}) are plotted in Figure \ref{solitonregionn}. From (\ref{soliton
region error}), the decay of the difference between $u_{\text{num}}\left(
x,t\right) $ and $u_{\text{sol}}^{1}\left( x,t\right) $ is sufficiently
rapid with respect to time such that two solutions can match very well at $%
t=80.$ However, in region $(ii)$ the decay rate is relatively slow judging
either from the previous result or from our numerical result in the
following subsection.

\begin{figure}[h]
\includegraphics[width=.8\textwidth]{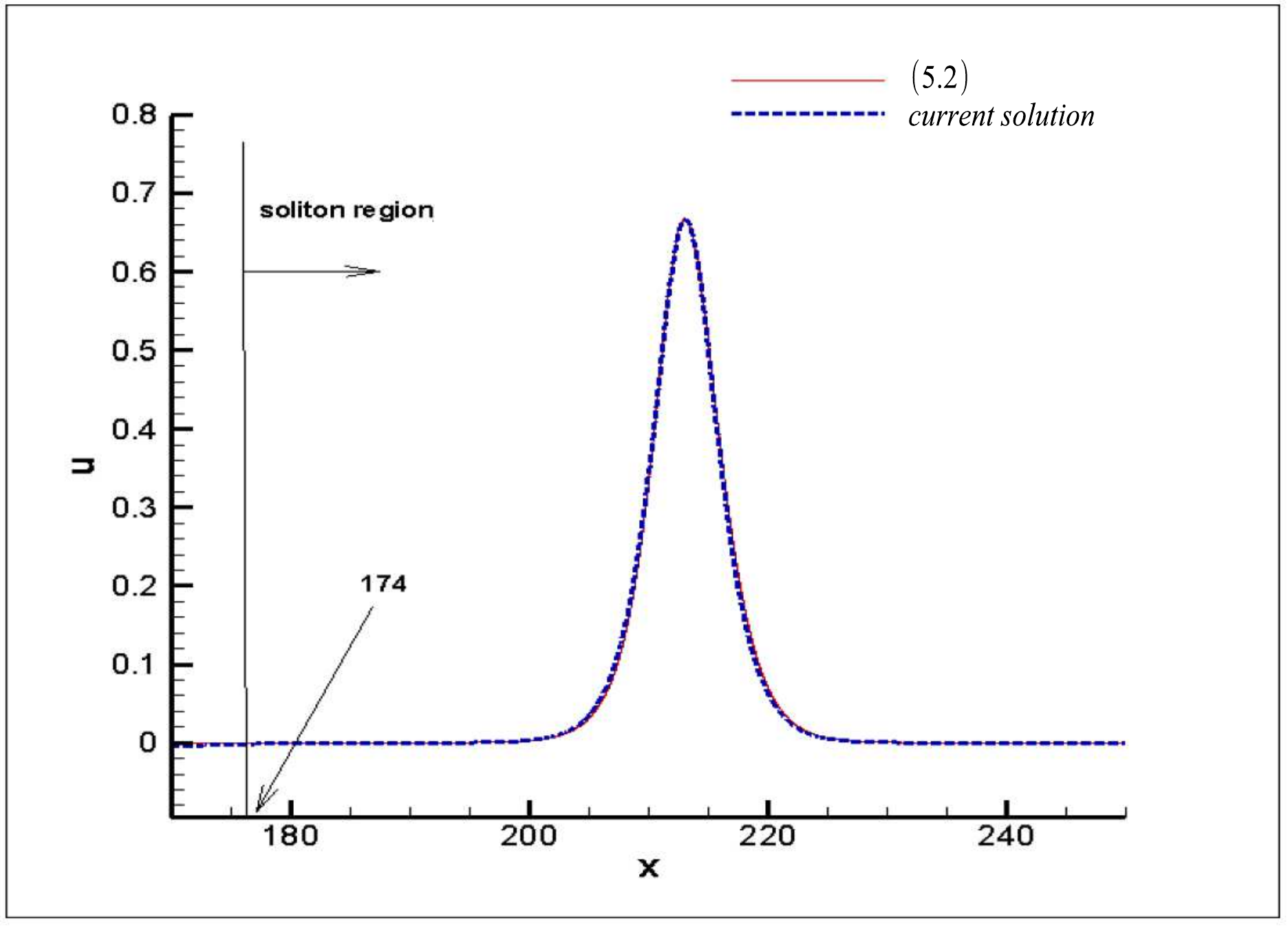}
\caption{{\protect \small Comparison of the currently predicted finite
difference solution and the asymptotic solution (\protect \ref{q=1/2 1soliton}%
) in the soliton region at }$t=80,${\protect \small \ }$\protect \varepsilon =%
\frac{14}{80}.$}
\label{solitonregionn}
\end{figure}

\subsubsection{Solution in the slowly decaying modulated oscillation region $%
(ii)$}

In the first oscillatory\ region $0\leq \zeta =\frac{x}{t}<2-\varepsilon ,$
the solution $u\left( x,t\right) $ takes the form as%
\begin{equation}
u\left( x,t\right) =\frac{c_{1}^{\left( 0\right) }}{\sqrt{t}}\sin \left(
c_{2}^{\left( 0\right) }t+c_{3}^{\left( 0\right) }\log t+c_{4}^{\left(
0\right) }\right) +O\left( t^{-\alpha }\right) ,
\label{oscillatory region (ii)}
\end{equation}%
for any $\alpha \in \left( \frac{1}{2},1\right) $ provided $l\geq 5.$ $%
c_{n}^{\left( 0\right) }$ ($n=1,...,4$) are functions of $\zeta \ $depending
on the reflection coefficient $R\left( k\right) $ (see \cite{MKST} or the
Appendix A).

At $t=40,80,160,200,$ and $\varepsilon =\frac{14}{80},$ the slowly decaying
modulated oscillation region (B) shown in Figure \ref{fourregionYU} is known
to be in the region of $0\leq x<73,0\leq x<146,0\leq x<292$ and $0\leq x<365$%
, respectively. The numerically predicted solution and the asymptotic
solution (\ref{oscillatory region (ii)}) are plotted in Figure \ref%
{YU2ndregionsimulation}.

By comparing the two formulas (\ref{soliton region error}) and (\ref%
{oscillatory region (ii)}), it is clear that the decay order of the computed
difference in the region $(i)$ is much larger than that in the region $(ii).$
Hence to get a good match between the two solutions in (\ref{oscillatory
region (ii)}), the time span needs to be longer than that in region $(i)$.
This can be examined from Figure \ref{YU2ndregionsimulation}. We can see
that at $t=80,$ although the numerical solution and the soliton solution in (%
\ref{q=1/2 1soliton}) match already very well in the soliton region $(i)$ in
Figure \ref{solitonregionn}, the numerical solution and the asymptotic
solution in (\ref{oscillatory region (ii)}) have not matched well in regions
with large $x$ in region $(ii)$ in Figure \ref{YU2ndregionsimulation}. We
need, therefore, a longer time for example $t=200$ seen in Figure \ref%
{YU2ndregionsimulation} to guarantee a good match between the numerical and
the asymptotic solutions.
\begin{figure}[h]
\includegraphics[width=.9\textwidth]{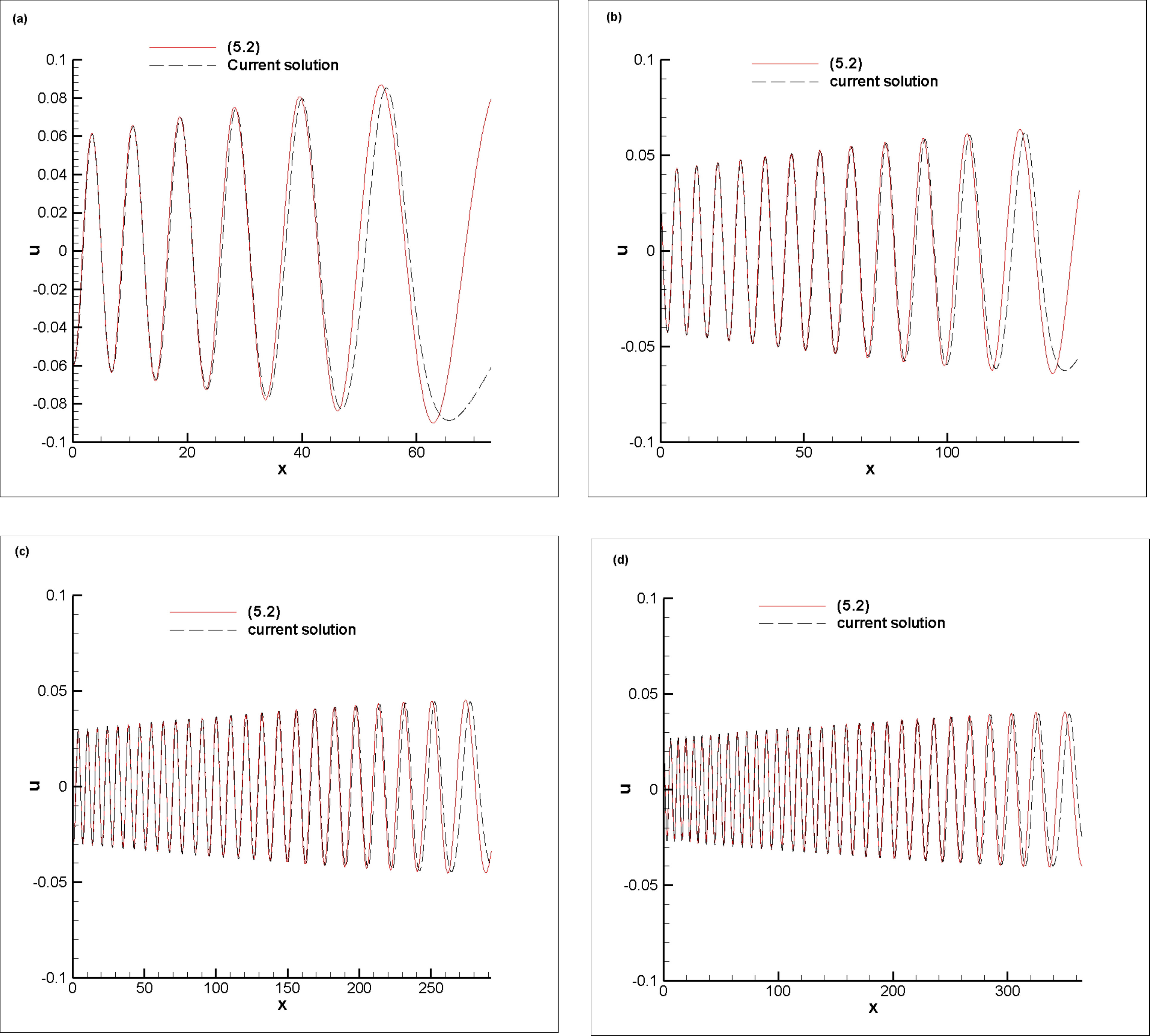}
\caption{{\protect \small Comparison of the predicted solution and the
asymptotic solution (\protect \ref{oscillatory region (ii)}) in the first
oscillatory region for }$\protect \varepsilon =\frac{14}{80}${\protect \small %
\ at (a) }$t=40;${\protect \small \ (b) }$t=80;${\protect \small \ (c) }$%
t=160; ${\protect \small \ (d) }$t=200.$}
\label{YU2ndregionsimulation}
\end{figure}

\subsubsection{Solution in region $(iii)$ with the sum of two decaying
modulated oscillations}

In the second oscillatory\ region $\frac{-1}{4}+\varepsilon <\zeta =\frac{x}{%
t}<0,$ $u\left( x,t\right) $ behaves like%
\begin{equation}
u\left( x,t\right) =\sum_{j=1}^{2}\frac{c_{1}^{\left( j\right) }}{\sqrt{t}}%
\sin \left( c_{2}^{\left( j\right) }t+c_{3}^{\left( j\right) }\log
t+c_{4}^{\left( j\right) }\right) +O\left( t^{-\alpha }\right) ,
\label{2nd asymp}
\end{equation}%
for any $\alpha \in \left( \frac{1}{2},1\right) $ provided $l\geq 5$, where $%
c_{n}^{\left( j\right) }$ ($n=1,...,4,$ $j=1,2$) are defined in the same way
as $c_{n}^{\left( 0\right) }$.

At $t=40,80,160,200,$ and $\varepsilon =\frac{14}{80}$, the slowly decaying
modulated oscillation region (C) shown in Figure \ref{fourregionYU} is known
to fall into $-3\leq x<0,-6\leq x<0,-12\leq x<0$ and $-15\leq x<0$,
respectively. The numerically predicted solution and the asymptotic solution
in (\ref{2nd asymp}) are plotted in Figure \ref{YU3rdregionsimulation}.
Similar to the solution in region $(ii)$, we can see from Figure \ref%
{YU3rdregionsimulation} that the numerical and asymptotic solutions become
well matched until $t=200$.
\begin{figure}[h]
\includegraphics[width=.9\textwidth]{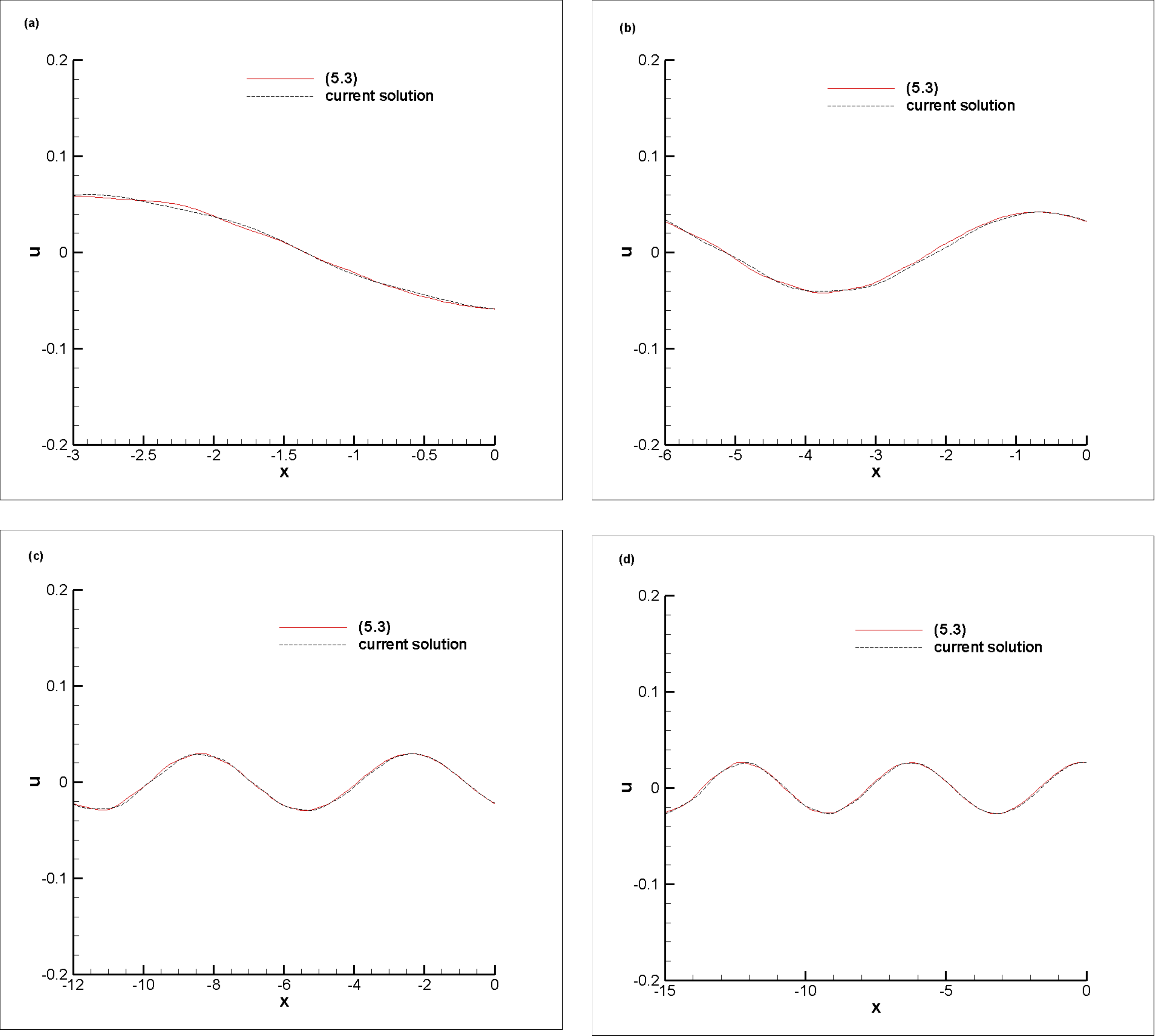}
\caption{{\protect \small Comparison of the currently predicted solution and
the asymptotic solution (\protect \ref{2nd asymp}) in the second oscillatory
region for }$\protect \varepsilon =\frac{14}{80}${\protect \small \ at (a) }$%
t=40;${\protect \small \ (b) }$t=80;${\protect \small \ (c) }$t=160;$%
{\protect \small \ (d) }$t=200.$}
\label{YU3rdregionsimulation}
\end{figure}

\subsubsection{Solution in the fast decaying region $(iv)$}

In the fast decaying region $\zeta =\frac{x}{t}<\frac{-1}{4}-\varepsilon ,$ $%
u\left( x,t\right) $ behaves like%
\begin{equation}
u\left( x,t\right) =O\left( t^{-l}\right) \text{ for any }l>0.
\label{fast decay asymp}
\end{equation}%
If we consider the case of $\varepsilon =\frac{14}{80},$ the fast decaying
region (D) schematically shown in Figure \ref{fourregionYU} at $%
t=40,80,160,200$ occurs in the regions of $x<-17,x<-34,x<-68,x<-85,$
respectively. The numerically predicted solution and the zero (exact)
solution are plotted in Figure \ref{iv}.

Although the decay order of error in (\ref{fast decay asymp}) is as large as
that in region $(i)$, by zooming the numerical and the asymptotic solution
profiles, we can find from Figure \ref{iv} the difference between the
regions $(i)$ and $(iv)$. At $t=40,$ the numerical solution is not very
close to zero for $x$ approaching the right end. Until $t=200,$ these two
solutions can match well with each other for the vertical solution axis
being scaled by $10^{-20}.$%
\begin{figure}[h]
\includegraphics[width=.9\textwidth]{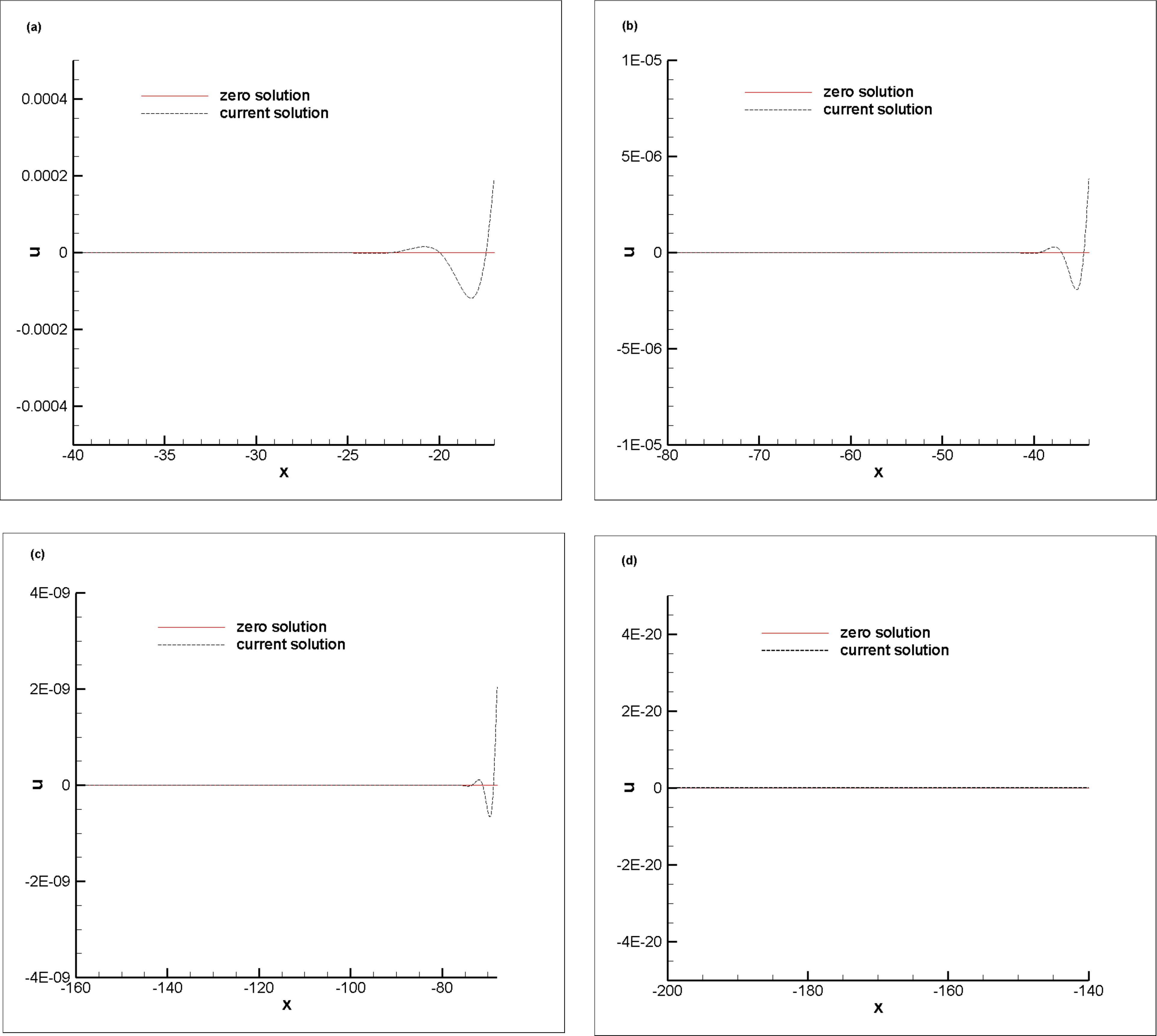}
\caption{{\protect \small Comparison of the predicted solution and the zero
(exact) solution in the fast decaying region for }$\protect \varepsilon =%
\frac{14}{80}${\protect \small \ at (a) }$t=40;${\protect \small \ (b) }$t=80;$%
{\protect \small \ (c) }$t=160;${\protect \small \ (d) }$t=200.$}
\label{iv}
\end{figure}

\subsection{Discussion of results in two Painlev\'{e} regions}

The Painlev\'{e} equation of type II is given as%
\begin{equation}
v^{\prime \prime }\left( s\right) =2v^{3}\left( s\right) +sv\left( s\right) .
\label{P2}
\end{equation}%
The one-parameter family of solutions to (\ref{P2}) was obtained firstly by
Ablowitz and Segur \cite{bib:Ablowitz(1977)} and then by Hastings and Mcleod
\cite{HasMcl1980} for $s\in
%TCIMACRO{\U{211d} }%
%BeginExpansion
\mathbb{R}
%EndExpansion
.$ They considered the solutions of (\ref{P2}) as the solutions of a linear
integral equation of the Gel'fand-Levitan type with a given $r\in
%TCIMACRO{\U{211d} }%
%BeginExpansion
\mathbb{R}
%EndExpansion
$%
\begin{equation*}
K\left( s,\eta \right) =r\text{Ai}\left( \frac{s+\eta }{2}\right) +\frac{%
r^{2}}{4}\int_{s}^{\infty }\int_{s}^{\infty }K\left( s,\bar{\eta}\right)
\text{Ai}\left( \frac{\bar{\eta}+\tilde{\eta}}{2}\right) \text{Ai}\left(
\frac{\tilde{\eta}+\eta }{2}\right) d\bar{\eta}d\tilde{\eta}\text{.}
\end{equation*}%
In the above, Ai$\left( s\right) $ is the Airy function. Then $v\left(
s\right) \left( =v\left( s;r\right) =K\left( s,s\right) \right) $ satisfies (%
\ref{P2}) subject to the specified boundary condition $v\left( s;r\right)
\sim r$Ai$\left( s\right) $ as $s\rightarrow \infty ,$ $s\in
%TCIMACRO{\U{211d} }%
%BeginExpansion
\mathbb{R}
%EndExpansion
$ (see also \cite{bib:Clarkson}).

\subsubsection{Solution in the first transition region (T1)}

From \cite{bib:Monvel(2010)}, for $\left \vert \frac{x}{t}-2\right \vert t^{%
\frac{2}{3}}<C,$ $u\left( x,t\right) $ behaves like%
\begin{equation}
u\left( x,t\right) =-\left( \frac{4}{3}\right) ^{\frac{2}{3}}\frac{1}{t^{%
\frac{2}{3}}}\left( v^{2}\left( s\right) -v^{\prime }\left( s\right) \right)
+O\left( t^{-1}\right) ,  \label{Painleve 1st reg}
\end{equation}%
where $s=6^{\frac{-1}{3}}\left( \frac{x}{t}-2\right) t^{\frac{2}{3}},$ and $%
v=v\left( s\right) =v\left( s;-R\left( 0\right) \right) $ is the real
valued, non-singular solution of (\ref{P2}) fixed by $v\left( s\right) \sim
-R\left( 0\right) $Ai$\left( s\right) $ as $s\rightarrow \infty $.\smallskip

In the numerical computation of the Painlev\'{e} equation of type II, we use
the connection formulas derived by \cite{ClarksonMcleod,HasMcl1980}. Since $%
R\left( 0\right) =-1$ from the equation (\ref{scatteringdata}), by the
Hastings and Mcleod's result \cite{HasMcl1980} the $P_{II}$ solution of
equation (\ref{P2}) sought subject to the boundary condition $v\left(
s\right) \sim -R\left( 0\right) Ai\left( s\right) =Ai\left( s\right) $ is
unique. Moreover, the asymptotic solution for $v$ can be expressed as $%
v\left( s\right) \sim \sqrt{\frac{-1}{2}s}$ as $s\rightarrow -\infty $.

By choosing a large interval, say $\left( s_{L},s_{R}\right) $ with $%
s_{L}s_{R}<0,$ the solution of (\ref{P2}) sought subject to the boundary
condition $v\left( s_{L}\right) =\sqrt{\frac{-1}{2}s_{L}},$ $v\left(
s_{R}\right) =Ai\left( s_{R}\right) =\frac{1}{2\sqrt{\pi }}\left(
s_{R}\right) ^{\frac{-1}{4}}e^{\frac{-2}{3}\left( s_{R}\right) ^{\frac{3}{2}%
}}\ $can be obtained. This solution is called as $v_{2}\left( s\right) .$

If the case with $C=\frac{14}{80}$ is considered, the first transition
region (T1) schematically shown in Figure \ref{fourregionYU} at $t=80$ and $%
160$ is defined in the regions of $159.25<x<160.75,$ $320.95<x<319.05,$
respectively. The numerically predicted solution and the leading term of the
asymptotic solution shown in (\ref{Painleve 1st reg}) with $v=v_{2}\left(
s\right) $ are plotted in Figure \ref{1stPainleve}.

We can then compare the leading term of the asymptotic solution (\ref%
{Painleve 1st reg}) with $v=v_{2}\left( s\right) $ and $u\left( x,t\right) $
at, for example, $t=40.$ This corresponds to compare the finite difference
solution $u\left( x,40\right) $ and the following term%
\begin{equation}
-\left( \frac{4}{3}\right) ^{\frac{2}{3}}\frac{1}{t^{\frac{2}{3}}}\left(
v_{2}^{2}\left( s\right) -v_{2}^{\prime }\left( s\right) \right) .
\label{P2 asymp in 1st p region}
\end{equation}%
Figure \ref{1stPainleve} exhibits a perfect agreement between the two graphs
for the numerical solution and the asymptotic solution (\ref{P2 asymp in 1st
p region}) expressed in terms of $v_{2}\left( s\right) $ at $t=160.$%
\begin{figure}[h]
\includegraphics[width=.9\textwidth]{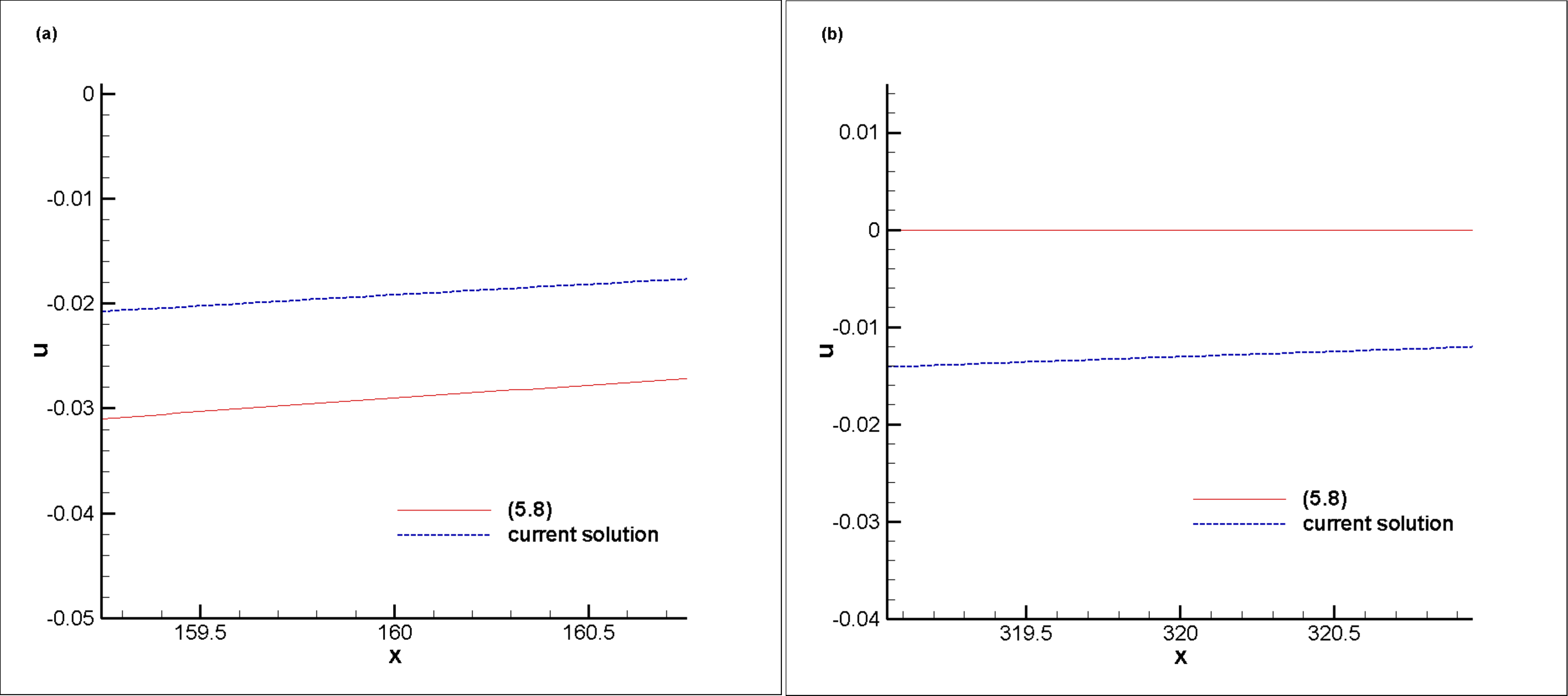}
\caption{{\protect \small Comparison of the predicted solution }$u\left(
x,t\right) ${\protect \small \ and the leading term of the Painlev\'{e}
solution (\protect \ref{Painleve 1st reg}) at (a) }$t=80${\protect \small ;
(b) }$t=160.$}
\label{1stPainleve}
\end{figure}

\subsubsection{Solution in the second transition region (T2)}

From \cite{bib:Monvel(2010)}, for $\left \vert \frac{x}{t}+\frac{1}{4}%
\right
\vert t^{\frac{2}{3}}<C,$ $u\left( x,t\right) $ has been derived as%
\begin{equation}
u\left( x,t\right) =\frac{12^{\frac{1}{6}}}{t^{\frac{1}{3}}}v_{1}\left(
s_{1}\right) \sin \psi \left( s_{1},t\right) +O\left( t^{-\frac{2}{3}%
}\right) ,  \label{Painleve 2nd reg}
\end{equation}%
where $s_{1}=-\left( \frac{16}{3}\right) ^{\frac{1}{3}}\left( \frac{x}{t}+%
\frac{1}{4}\right) t^{\frac{2}{3}}$ and $v_{1}\left( s\right) =v_{1}\left(
s;\left \vert R\left( \frac{\sqrt{3}}{2}\right) \right \vert \right) $ is
the real valued, non-singular solution of (\ref{P2}) fixed by $v_{1}\left(
s\right) \sim \left \vert R\left( \frac{\sqrt{3}}{2}\right) \right \vert
Ai\left( s\right) $ as $s\rightarrow \infty $. Other notations shown in (\ref%
{Painleve 2nd reg}) and their $q_{0}$-dependence can be seen in the Appendix
A.

Since $\left \vert R\left( \frac{\sqrt{3}}{2}\right) \right \vert <1$, by
the Clarkson and Mcleod's result \cite{ClarksonMcleod} the solution of the $%
P_{II}$ equation (\ref{P2}) obtained under the boundary condition $v\left(
s\right) \sim \left \vert R\left( \frac{\sqrt{3}}{2}\right) \right \vert $Ai$%
\left( s\right) $ is unique. Moreover, the corresponding asymptotic behavior
for $v$ as $s\rightarrow -\infty $ can be expressed as%
\begin{equation}
v\left( s\right) \sim d\left \vert s\right \vert ^{\frac{-1}{4}}\sin \left
\{ \frac{2}{3}\left \vert s\right \vert ^{\frac{3}{2}}-\frac{3}{4}d^{2}\ln
\left \vert s\right \vert -\theta _{0}\right \} \text{ as }s\rightarrow
-\infty ,  \label{Clarkson Mcleod formula}
\end{equation}%
where $d$ and $\theta _{0}$ can be seen in \cite{ClarksonMcleod} and their $%
q_{0}$-dependence can be seen in the Appendix A.

By choosing a large interval, say $\left( \hat{s}_{L},\hat{s}_{R}\right) $
with $\hat{s}_{L}\hat{s}_{R}<0,$ the solution of (\ref{P2}) with the
boundary condition $v\left( \hat{s}_{L}\right) =B_{-}\left( \hat{s}%
_{L}\right) ,$ $v\left( \hat{s}_{R}\right) =B_{+}\left( \hat{s}_{R}\right) $
(see (\ref{connection formula}) in the Appendix A) can be obtained. We call
this solution as $v_{3}\left( s\right) .$

At $C=\frac{14}{80}$, the second transition region (T2) schematically shown
in Figure \ref{fourregionYU} at $t=80$ and $160$ falls in the regions of $%
-20.754<x<-19.246$ and $-40.95<x<-39.05,$ respectively. The numerically
predicted solution and the leading term of the asymptotic solution taking
the form of (\ref{Painleve 2nd reg}) with $v=v_{3}\left( s\right) $ are
plotted in Figure \ref{p2}.

Based on the similar argument to the one presented in the first Painlev\'{e}
region, the predicted solution $u\left( x,t\right) $ can be now compared
with the leading term (\ref{Painleve 2nd reg}) of the Painlev\'{e} II
solution%
\begin{equation}
\frac{12^{\frac{1}{6}}}{t^{\frac{1}{3}}}v_{3}\left( s_{1}\right) \sin \psi
\left( s_{1},t\right) \text{.}  \label{P2 asymp in 1st p region_}
\end{equation}%
In Figure \ref{p2}, at $t=160$ we can see a good agreement between the two
graphs of the numerical solution and the asymptotic solution (\ref{P2 asymp
in 1st p region_}) expressed in terms of $v_{3}\left( s\right) $.

\begin{figure}[h]
\includegraphics[width=.9\textwidth]{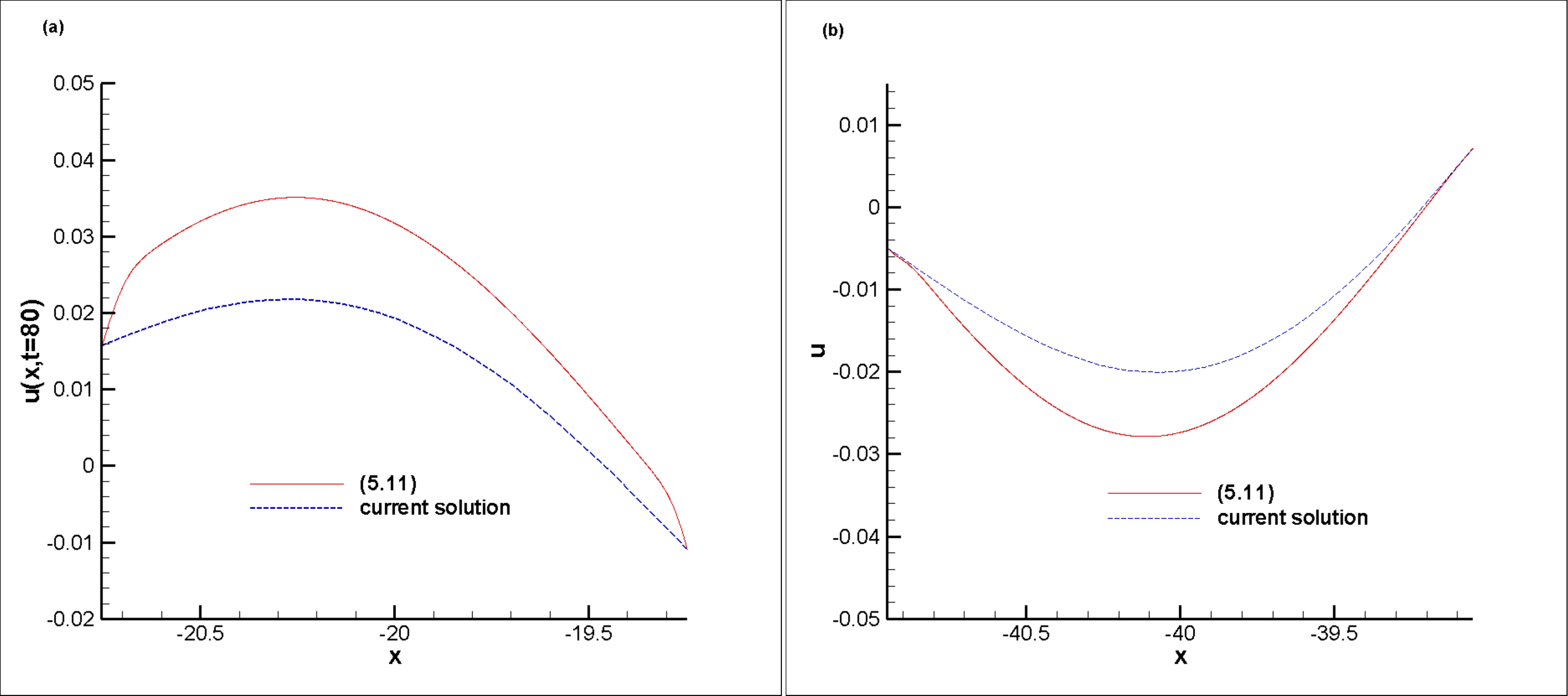}
\caption{{\protect \small Comparison of the predicted solution }$u\left(
x,t\right) ${\protect \small \ and the leading term of the Painlev\'{e} II
solution (\protect \ref{Painleve 2nd reg}), or }$v_{3}\left( s\right) ,$%
{\protect \small \ at (a) }$t=80${\protect \small , (b) }$t=160.$}
\label{p2}
\end{figure}

\subsection{The $\protect \varepsilon -$ and $C-$ dependence in the
asymptotic regions}

Both $\varepsilon $ and $C$ appeared in the beginning of this section can
characterize the regions of different solution behaviors. In the $\left(
x,t>0\right) $ plane, the $\varepsilon $ appeared in the four regions $(i)$
- $(iv)$ characterizes the degree of the departure from the two lines $\frac{%
x}{t}=2$ and $\frac{x}{t}=\frac{-1}{4}$. The constant $C$ present in the two
transition regions characterizes the degree of departure from the lines $%
\frac{x}{t}=2$ and $\frac{x}{t}=\frac{-1}{4}.$

In Figure \ref{eplison_C} we plot the variance of the asymptotic regions for
$C=\frac{14}{80}$ and $\varepsilon =\frac{10}{80},\frac{14}{80}$ and $\frac{%
18}{80}$, respectively. From the result in \cite{MKST} (see the Appendix A),
let us take the regions $(iii)$ and $(iv)$ as an example for the discussion
of results in Figure \ref{eplison_C}. For $C$ being fixed at $\frac{14}{80}$%
, given a large enough $\varepsilon $, the regions $(iii)$ $\frac{-1}{4}%
+\varepsilon <\frac{x}{t}<0$ and $(iv)$ $\frac{x}{t}<\frac{-1}{4}%
-\varepsilon $ will exhibit oscillatory asymptotics. However, as $%
\varepsilon $ becomes smaller, regions $(iii)$ and $(iv)$ become larger, and
at the same time regions $(iii)$ and $(iv)$ are closer to the second
transition region $\left \vert \frac{x}{t}+\frac{1}{4}\right \vert t^{\frac{2%
}{3}}<C$. As long as $t$ is sufficiently large, the condition $Ct^{\frac{-2}{%
3}}<\varepsilon $ can be satisfied no matter what the chosen values of $%
\varepsilon $ and $C$ are. That is, the transition region (T2) will not
affect the regions $(iii)$ and $(iv)$ for $t$ being larger enough for any
choice of $\varepsilon $ and $C.$ One should notice that in figure (c) of
Figure \ref{eplison_C}, the region with white color between the region $%
(iii) $ and the region (T2) can not be considered as a region belonging to
neither the region $(iii)$ nor the region (T2). In fact, this white region
should be regarded as the region $(iii)$ with $\varepsilon $ becoming
greater than $\frac{18}{80}.$ Other white regions in different places can be
argued by a similar way.

\begin{figure}[h]
\includegraphics[width=.8\textwidth]{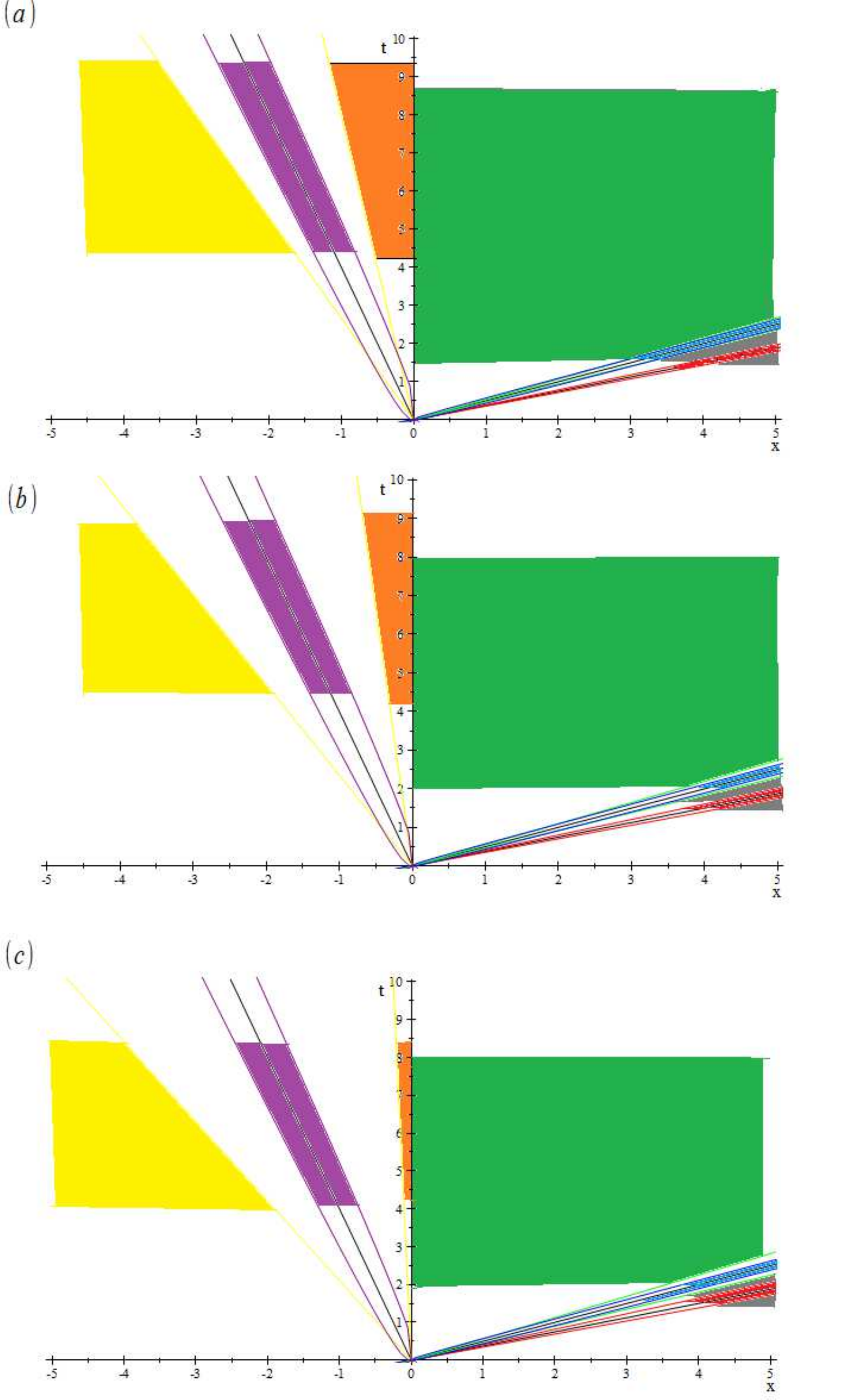}
\caption{{\protect \small Comparison of the regions of long-time asymptotics
obtained at }$C=\frac{14}{80}${\protect \small \ and (a) }$\protect%
\varepsilon =\frac{10}{80};${\protect \small \ (b)}$\frac{14}{80};$%
{\protect \small \ (c) }$\frac{18}{80}.${\protect \small \ The first (resp.
second) transition region is colored with blue (resp. purple).}}
\label{eplison_C}
\end{figure}

\subsection{Discussion of the time-decay estimates \label{(Q1)(Q2)}}

In questions (Q1) and (Q2), we are asked ourselves how to determine the time
that is long enough that the predicted solution has matched with asymptotic
solution and how to obtain the powers of $t$ in the time-decay estimates
under the specified initial data (\ref{ini15}). In this section, we are
aimed to answer these two questions in the following way. Let $u_{j}\left(
x,t\right) $ with $j=$1,...,4, 5 and 6 be the asymptotic solutions in
regions $(i)$,...,$(iv)$ and in the first and the second transition regions,
respectively. Let $I_{j}$ with $j=$1,...,4, 5 and 6 denoting the regions $%
(i) $,..., $(iv)$ and the first and the second transition regions,
respectively.

In Subsection \ref{(Q1)} we find the time $T_{j}$ in each $I_{j}$ such that
the difference between the numerical solution $u_{\text{num}}\left(
x,t\right) $ and the asymptotic solution $u_{j}\left( x,t\right) $ becomes
small enough for $t\geq T_{j}.$ Here "small enough" between $u_{\text{num}%
}\left( x,t\right) $ and $u_{j}\left( x,t\right) $ solutions is measured by
considering the $L^{2}$ norm difference between $u_{\text{num}}\left(
x,t\right) $ and $u_{j}\left( x,t\right) $ in each region $I_{j}:$%
\begin{equation*}
E_{j}^{L2}\left( t\right) :=\left( \int_{I_{j}}\left \vert u_{\text{num}%
}\left( x,t\right) -u_{j}\left( x,t\right) \right \vert ^{2}dx\right) ^{%
\frac{1}{2}}
\end{equation*}%
for $j=$1,...,4, 5 and 6. We find the time $T_{j}$ in each $I_{j}$ such that
$E_{j}^{L2}\left( t\right) $ becomes small and there is no significant
variation any longer for $t\geq T_{j}.$ Then the time needed for the
computed solutions to reach the asymptotics can be found in each region $%
I_{j}$.

In Subsection \ref{(Q2)}, the decay order theoretically derived from \cite%
{MKST} will be calculated in a numerical way. Let $E_{j}^{\sup }\left(
t\right) :=\max_{x\in I_{j}}\left \vert u_{\text{num}}\left( x,t\right)
-u_{j}\left( x,t\right) \right \vert .$ Based on the point-wise sense of the
asymptotic formulas (\ref{soliton region error}), (\ref{oscillatory region
(ii)}), (\ref{2nd asymp}) and (\ref{fast decay asymp}) in \cite{MKST}, the
function space considered for (Q2) is different from that in Subsection \ref%
{(Q1)}. The exact power of $t$ in these asymptotic formulas remains unknown
in \cite{MKST}. The above mentioned question (Q2) concerns the way of
computing the value of the power of $t.$

In Subsection \ref{eplison influence rate}, we discuss the influence of the
range considered in each region on the computed decay rates. The result
shows that if a larger region is considered, the time needed to reach the
long-time asymptotics will be longer accordingly. That is, the computed
decay rates will become smaller.

\subsubsection{The predicted long time in each region \label{(Q1)}}

We find the time needed to reach the asymptotics by computing $%
E_{j}^{L2}\left( t\right) $ in each $I_{j}$ ($j=2$, $3$, $4$).
\begin{figure}[tbh]
\includegraphics[width=.9\textwidth]{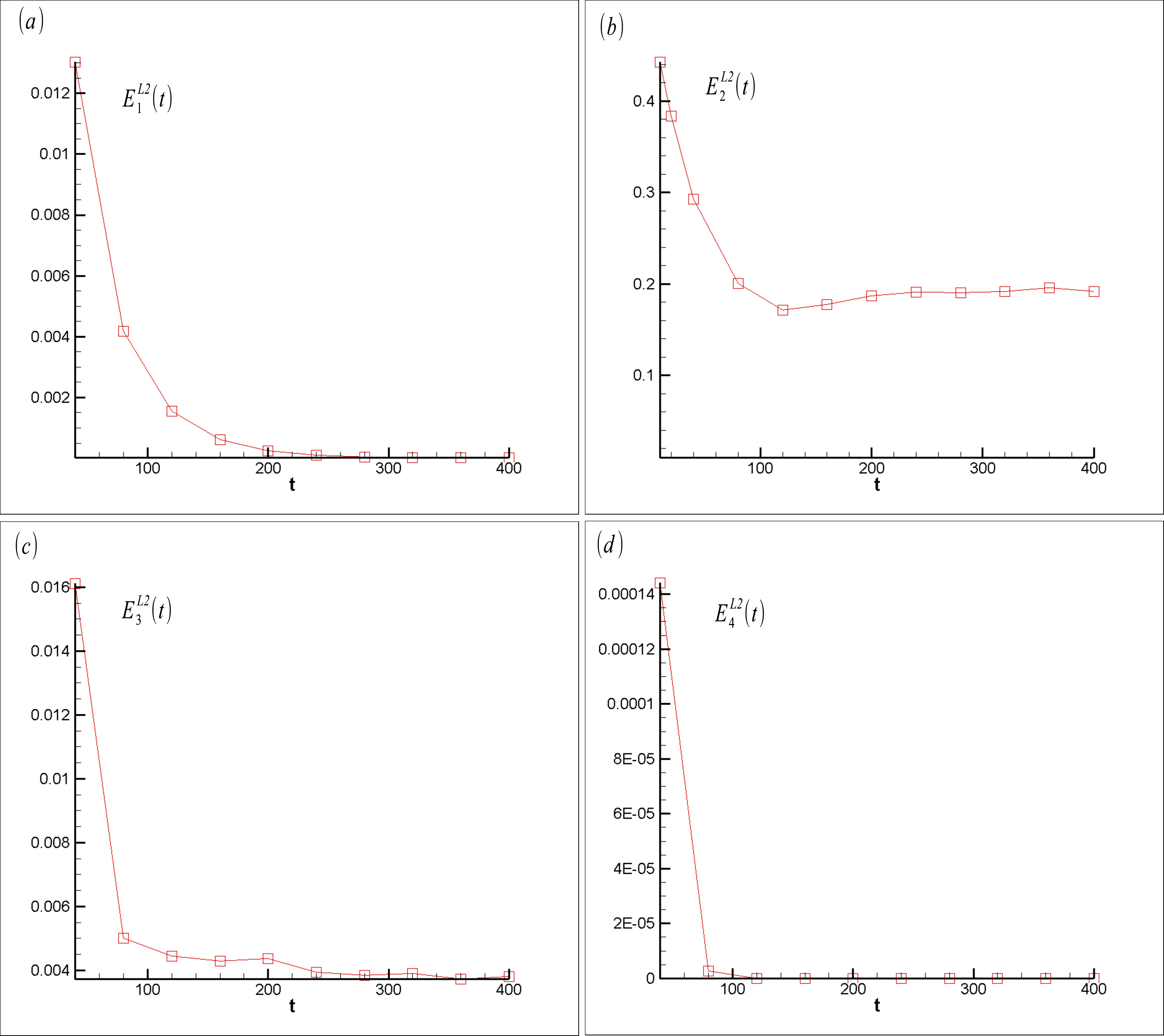}
\caption{{\protect \small The predicted }$L^{2}${\protect \small \ norm
differences between the numerical and long-time asymptotic solutions are
plotted with respect to time in the regions (a) }$(i);${\protect \small \ (b)
}$(ii);${\protect \small \ (c) }$(iii);${\protect \small \ (d) }$(iv)$%
{\protect \small \ for }$\protect \varepsilon =\frac{14}{80}.$}
\label{L24regions}
\end{figure}
Figure \ref{L24regions} plots the time-decay profiles of $E_{j}^{L2}\left(
t\right) $ in regions $(i)$-$(iv)$ at different times with $\varepsilon =%
\frac{14}{80}$. In region $(iii)$, for a large enough $t$, the estimate $%
E_{3}^{L2}\left( t\right) $ becomes sufficiently small. In regions $(i)$ and
$\left( iv\right) $, the decay orders are much larger than those in regions $%
(ii)$ and $(iii)$ from these graphs. Moreover, the variances of $%
E_{j}^{L2}\left( t\right) $ ($j=1,$ $2$, $3$, $4$) are tabulated in Table %
\ref{L2 i}.

\begin{table}[tbh]
{\small
\begin{equation*}
\begin{tabular}{|l|l|l|l|l|}
\hline
$t$ & $E_{1}^{L2}(t)$ & $E_{2}^{L2}\left( t\right) $ & $E_{3}^{L2}\left(
t\right) $ & $E_{4}^{L2}\left( t\right) $ \\ \hline
$40$ & 1.303304031375658E-2 & 0.292954164986533 & 1.612163792174788E-2 &
1.442232517777524E-4 \\ \hline
$80$ & 4.170394908225294E-3 & 0.200463612016154 & 5.009588123729806E-3 &
2.575958816177057E-6 \\ \hline
$120$ & 1.557160720783265E-3 & 0.171595918602204 & 4.443772185385986E-3 &
5.457070673560806E-8 \\ \hline
$160$ & 6.124634754925668E-4 & 0.177361156697756 & 4.289889498025333E-3 &
1.163202136797550E-9 \\ \hline
$200$ & 2.498850932347651E-4 & 0.187198259978890 & 4.372113314984243E-3 &
2.755050865021234E-11 \\ \hline
$240$ & 1.060794858467605E-4 & 0.191131350654329 & 3.944857443413838E-3 &
6.595925065534878E-13 \\ \hline
$280$ & 5.032093370750756E-5 & 0.190235716309735 & 3.837028523022252E-3 &
1.572983122931371E-14 \\ \hline
$320$ & 3.113956434979408E-5 & 0.191477176087681 & 3.890293609233680E-3 &
3.709034000389790E-16 \\ \hline
$360$ & 2.655384864147633E-5 & 0.195600653012653 & 3.727629679179258E-3 &
8.614729075722736E-18 \\ \hline
$400$ & 2.575663036324579E-5 & 0.191706872966999 & 3.802122218224038E-3 &
1.966797275673494E-19 \\ \hline
\end{tabular}%
\end{equation*}%
}
\caption{{\protect \small The differences }$E_{j}^{L2}\left( t\right) $%
{\protect \small \ (}$j=1,2,3,4${\protect \small ) obtained in the regions }$%
(i)-(iv)${\protect \small \ at different times for }$\protect \varepsilon =%
\frac{14}{80}${\protect \small .}}
\label{L2 i}
\end{table}

Motivated by the fact given in \cite{MKST} that the decay orders in regions $%
(ii)$ and $(iii)$ are smaller than those in regions $(i)$ and $(iv),$ we
focus on the closeness between the numerical and the asymptotic solutions in
regions $(ii)$ and $(iii).$ That is, in a shorter time, the analytical
solutions already match the asymptotic solutions very well in regions $(i)$
and $(iv).$ Therefore $\min \left \{ T_{2},T_{3}\right \} >\max \left \{
T_{1},T_{4}\right \} $ and the long time $T:=\max_{1\leq j\leq 4}T_{j}$
stated above needs to be found by examining $T_{2}$ and $T_{3}$ in regions $%
(ii)$ and $(iii)$, respectively.

In Table \ref{L2 i} and Figure \ref{L24regions}, we take $T_{1}=160,$ $%
T_{3}=240$ and $T_{4}=40$ as the respective times at which the numerical
solutions match very well with the asymptotic solutions in regions $(i),$ $%
(iii)$ and $(iv)$. However, $T_{2}$ can not determined easily from the graph
in Figure \ref{L24regions}. We choose $T_{2}=280$ and the reason will be
explained in the next section.

\subsubsection{The numerically computed decay powers of $t$\label{(Q2)}}

In this subsection we consider the powers of $t$ of the time-decay estimates
in \cite{MKST}.
\begin{figure}[tbh]
\includegraphics[width=.9\textwidth]{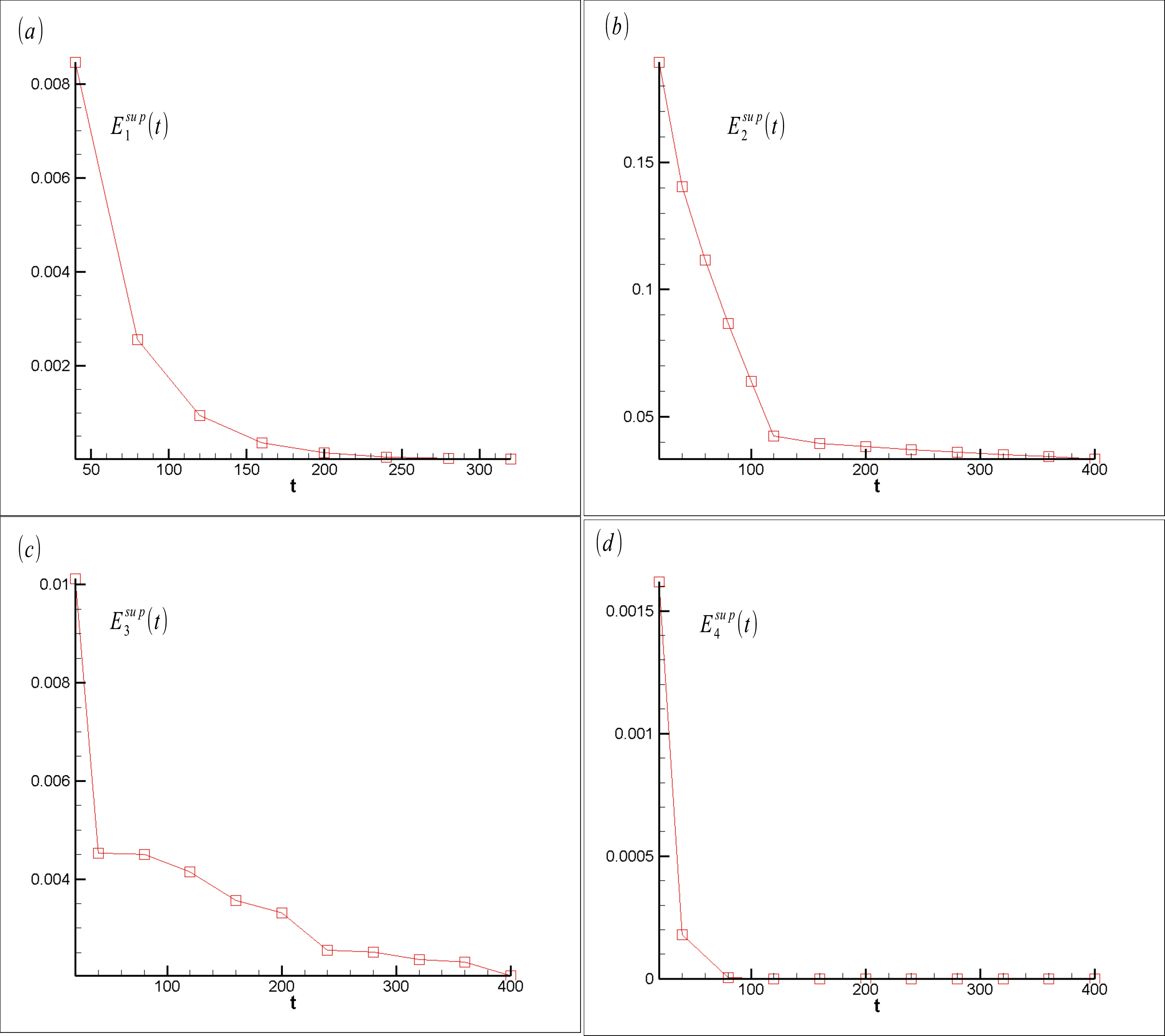}
\caption{{\protect \small The predicted sup norm differences between the
numerical and long-time asymptotic solutions are plotted with respect to
time in the regions (a) }$(i)${\protect \small \ (b) }$(ii);${\protect \small %
\ (c) }$(iii);${\protect \small \ (d) }$(iv)${\protect \small \ for }$\protect%
\varepsilon =\frac{14}{80}.$}
\label{Esupp}
\end{figure}
Figure \ref{Esupp} plots the time-decay profiles of the sup norm $%
E_{j}^{\sup }\left( t\right) $ in regions $(i)$-$(iv)$, respectively, at
different times with $\varepsilon =\frac{14}{80}$. Similar to the $L^{2}$
norm results, in regions $(ii)$ and $(iii)$ the estimates $E_{2}^{\sup
}\left( t\right) $ and $E_{3}^{\sup }\left( t\right) $ become small enough
for a large enough $t$. The decay orders in regions $(i)$ and $\left(
iv\right) $ look much larger than those in regions $(ii)$ and $(iii)$.
Moreover, the variances of $E_{j}^{\sup }\left( t\right) $ ($j=1,$ $2$, $3$,
$4$) are tabulated in Table \ref{sup i}.

\begin{table}[tbh]
{\small
\begin{equation*}
\begin{tabular}{|l|l|l|l|l|}
\hline
$t$ & $E_{1}^{\sup }(t)$ & $E_{2}^{\sup }\left( t\right) $ & $E_{3}^{\sup
}\left( t\right) $ & $E_{4}^{\sup }\left( t\right) $ \\ \hline
$40$ & 8.468787040579530E-3 & 1.40463349496939E-1 & 4.526819459939539E-2 &
1.802887487030715E-4 \\ \hline
$80$ & 2.562851596378625E-3 & 8.675408665040811E-2 & 4.496098688352251E-3 &
3.742290026355494E-6 \\ \hline
$120$ & 9.354531544216522E-4 & 4.260921144987524E-2 & 4.144216057718784E-3 &
8.745510811409106E-8 \\ \hline
$160$ & 3.636038798000425E-4 & 3.965394156144474E-2 & 3.566372610309353E-3 &
1.911307301562621E-9 \\ \hline
$200$ & 1.467776857832571E-4 & 3.840181001023902E-2 & 3.306124437020473E-3 &
4.524112000962450E-11 \\ \hline
$240$ & 6.063916391485836E-5 & 3.724938756956186E-2 & 2.548191267527537E-3 &
1.048626564931299E-12 \\ \hline
$280$ & 2.565702741164118E-5 & 3.619659640699480E-2 & 2.513406195082118E-3 &
2.373955192058340E-14 \\ \hline
$320$ & 1.755075330411859E-5 & 3.523470929795400E-2 & 2.352826688157257E-3 &
5.235427733097183E-16 \\ \hline
$360$ &  & 3.435271074529302E-2 & 2.311665267575798E-3 &
1.120502174844415E-17 \\ \hline
$400$ &  & 3.354081711805757E-2 & 2.030974658791331E-3 &
2.318829978373335E-19 \\ \hline
\end{tabular}%
\end{equation*}%
}
\caption{{\protect \small The differences }$E_{j}^{\sup }\left( t\right) $%
{\protect \small \ (}$j=1,2,3,4${\protect \small ) obtained in the regions }$%
(i)-(iv)${\protect \small \ at different times for }$\protect \varepsilon =%
\frac{14}{80}${\protect \small .}}
\label{sup i}
\end{table}

In the following the question (Q2) will be answered by two means. In other
words, the decay order of the estimate $E_{j}^{\sup }\left( t\right) $ will
be found numerically under two different error estimates:

(1) The rates of convergence (R.O.C.) in regions $(ii)$-$(iv)\ $are
computed. The results are summarized in Table \ref{ROC}.
\begin{table}[h]
\begin{equation*}
\begin{tabular}{|l||l||l|}
\hline
region $(ii)$ & region $(iii)$ & region $(iv)$ \\ \hline
0.7998085 & 0.5159832 & 13.1137 \\ \hline
\end{tabular}%
\end{equation*}%
\caption{{\protect \small The computed rates of convergence in the regions }$%
(ii)${\protect \small , }$(iii)${\protect \small \ and }$(iv)${\protect \small .%
}}
\label{ROC}
\end{table}

(2) In \cite{MKST}, the decay orders in the regions $\left( i\right) $ and $%
(iv)$ are $O\left( t^{-l}\right) \ $while $O\left( t^{-\alpha }\right) $ in
the regions $(ii)$ and $(iii)$ for any $\alpha \in \left( \frac{1}{2}%
,1\right) $ provided $l\geq 5$ (see (\ref{soliton region error}), (\ref%
{oscillatory region (ii)}), (\ref{2nd asymp}) and (\ref{fast decay asymp})).
Therefore there exist four different constants $C_{i}$ such that%
\begin{equation}
\begin{tabular}{l}
$E_{j}^{\sup }\left( t\right) \leq C_{j}t^{-l_{j}}$ for $j=1,4;\smallskip $
\\
$E_{j}^{\sup }\left( t\right) \leq C_{j}t^{-\alpha _{j}}$ for $j=2,3.$%
\end{tabular}
\label{estimates in mkst}
\end{equation}%
Note that $C_{j}$ $(j=1,...,4)$ are independent of $t$. In the following we
numerically calculate the decay orders in the regions $(i)$-$(iv)$.To this
end, we first define the numerically computed decay powers in the regions $%
(i)$-$(iv).$

\begin{definition}
\label{NDR}Let $T=\max \left \{ T_{1},...,T_{4}\right \} $ be the maximum of
the long times $T_{1},...,T_{4}$ found in question (Q1). The numerical decay
powers $L_{1},$ $L_{4}$, $A_{2}$ and $A_{3}$ are defined to be the values
such that%
\begin{equation*}
E_{j}^{\sup }\left( T\right) =T^{-L_{j}},\text{ }j=1,4;\text{ \ }E_{j}^{\sup
}\left( T\right) =T^{-A_{j}},\text{ }j=2,3.
\end{equation*}%
That is,%
\begin{equation*}
L_{1,4}=-\frac{\log E_{1,4}^{\sup }\left( T\right) }{\log T},\ A_{2,3}=-%
\frac{\log E_{2,3}^{\sup }\left( T\right) }{\log T}.
\end{equation*}%
The values $L_{1},$ $L_{4}$, $A_{2}$ and $A_{3}$ are considered to be the
powers of $t$ of the time-decay estimates measured by the sup norm between
the numerical solution and the asymptotic solution.
\end{definition}

One of the difficulties of computing the decay orders is that the values $%
C_{j}$ in (\ref{estimates in mkst}) can not be found theoretically. As a
result, we suppose $C_{j}=1$ in Definition \ref{NDR}.

From the results in (Q1), we have $T=\max \left \{ T_{1},...,T_{4}\right \}
=280.$ Then from Table \ref{sup i}, $L_{1},$ $L_{4}$, $A_{2}$ and $A_{3}$
can be found as%
\begin{equation*}
L_{1}=-\frac{\log E_{1}^{\sup }\left( T\right) }{\log T}=\frac{-\log \left(
6.063916391485836\times 10^{-5}\right) }{\log 240}=1.7718,
\end{equation*}%
\begin{equation*}
A_{2}=-\frac{\log E_{2}^{\sup }\left( T\right) }{\log T}=\frac{-\log \left(
3.619659640699480\times 10^{-3}\right) }{\log 280}=0.99762,
\end{equation*}%
\begin{equation*}
A_{3}=-\frac{\log E_{3}^{\sup }\left( T\right) }{\log T}=\frac{-\log \left(
2.513406195082118\times 10^{-3}\right) }{\log 280}=1.0623,
\end{equation*}%
\begin{equation*}
L_{4}=-\frac{\log E_{4}^{\sup }\left( T\right) }{\log T}=\frac{-\log \left(
2.373955192058340\times 10^{-14}\right) }{\log 280}=5.5675.
\end{equation*}

\noindent \textbf{Remark} According to the results, we can see that the
numerical decay powers match the restrictions of the powers of $t$ in
regions $(ii)$ and $(iv)$ from \cite{MKST}. The reason for choosing $%
T_{2}=280$ is therefore explained. We can further compute $A_{2}$ at
different times by%
\begin{equation*}
\left. A_{2}\right \vert _{t=240}=-\frac{\log E_{2}^{\sup }\left( 240\right)
}{\log 240}=\frac{-\log \left( 3.724938756956186\times 10^{-3}\right) }{\log
240}=1.0204.
\end{equation*}%
For the determination of $T$ in the question (Q1), Figure \ref{L24regions}
is not sufficient to answer this question. We need the predicted numerical
decay power since at $t=280$ the corresponding $A_{2}$ can match the result $%
\alpha \in \left( \frac{1}{2},1\right) $ in \cite{MKST}. Moreover, in region
$(iii),$ at $t=400,$ one can find that $A_{3}=\frac{-\log \left(
2.030974658791331\times 10^{-3}\right) }{\log 400}=1.0347.$ This value is
close to the power of $t$ in (\ref{2nd asymp}).

\subsubsection{The dependence of the numerical decay powers on $\protect%
\varepsilon $ \label{eplison influence rate}}

In order to compare the influence of the chosen $\varepsilon $ on $%
E_{j}^{\sup }\left( t\right) ,$ the variances of $E_{j}^{\sup }\left(
t\right) $ ($j=2$, $4$) obtained at different $\varepsilon $ are tabulated
in Table \ref{E2}.

\begin{table}[tbh]
\begin{equation*}
\begin{tabular}{|l|l|l|l|l|}
\hline
& $t$ & $\varepsilon =\frac{9}{80}$ & $\varepsilon =\frac{14}{80}$ & $%
\varepsilon =\frac{19}{80}$ \\ \hline \hline
& $40$ & 1.418562E-1 & 1.404633E-1 & 1.180903E-1 \\ \cline{2-5}
$E_{2}^{\sup }\left( t\right) $ & $80$ & 1.048379E-1 & 8.67541E-2 &
4.211969E-2 \\ \cline{2-5}
& $120$ & 8.665302E-2 & 4.26092E-2 & 4.09646E-2 \\ \cline{2-5}
& $160$ & 7.31637E-2 & 3.96539E-2 & 3.96539E-2 \\ \cline{2-5}
& $200$ & 5.899846E-2 & 3.84018E-2 & 3.84018E-2 \\ \cline{2-5}
& $240$ & 4.469613E-2 & 3.724938E-2 & 3.724938E-2 \\ \cline{2-5}
& $280$ & 3.619659E-2 & 3.61965E-2 & 3.619659E-2 \\ \cline{2-5}
& $320$ & 3.523470E-2 & 3.52347E-2 & 3.52347E-2 \\ \hline \hline
& $40$ & 1.319865E-3 & 1.80288E-4 & 2.801281E-5 \\ \cline{2-5}
$E_{4}^{\sup }\left( t\right) $ & $80$ & 1.783843E-4 & 3.74229E-6 &
6.24902E-8 \\ \cline{2-5}
& $120$ & 1.783843E-5 & 8.745510E-8 & 5.18444E-10 \\ \cline{2-5}
& $160$ & 1.783843E-6 & 1.91130E-9 & 6.47375E-13 \\ \cline{2-5}
& $200$ & 1.783843E-7 & 4.52411E-11 & 5.59605E-15 \\ \cline{2-5}
& $240$ & 1.783843E-8 & 1.04862E-12 & 6.87049E-18 \\ \cline{2-5}
& $280$ & 1.783843E-9 & 2.373955E-14 & 2.463507E-20 \\ \cline{2-5}
& $320$ & 1.783843E-10 & 5.235427E-16 & 8.57598E-24 \\ \hline
\end{tabular}%
\end{equation*}%
\caption{{\protect \small The differences }$E_{2}^{\sup }\left( t\right) $%
{\protect \small \ and }$E_{4}^{\sup }\left( t\right) ${\protect \small \
obtained in the regions }$(ii)${\protect \small \ and }$(iv)${\protect \small %
, respectively at different times for }$\protect \varepsilon =\frac{9}{80},$%
{\protect \small \ }$\protect \varepsilon =\frac{14}{80},${\protect \small \ }$%
\protect \varepsilon =\frac{19}{80}${\protect \small .}}
\label{E2}
\end{table}

\begin{table}[tbh]
\begin{equation*}
\begin{tabular}{|l|l|l|l|}
\hline
$t$ & $A_{2}$ at $\varepsilon =\frac{9}{80}$ & $A_{2}$ at $\varepsilon =%
\frac{14}{80}$ & $A_{2}$ at $\varepsilon =\frac{19}{80}$ \\ \hline
$40$ & $0.52941$ & $0.53209$ & $0.57912$ \\ \hline
$80$ & $0.51468$ & $0.55789$ & $0.72278$ \\ \hline
$120$ & $0.51088$ & $0.65915$ & $0.66737$ \\ \hline
$160$ & $0.51526$ & $0.63595$ & $0.63595$ \\ \hline
$200$ & $0.53418$ & $0.61522$ & $0.61522$ \\ \hline
$240$ & $0.56706$ & $0.60032$ & $0.60032$ \\ \hline
$280$ & $0.58898$ & $0.58898$ & $0.58898$ \\ \hline
$320$ & $0.58002$ & $0.51139$ & $0.58002$ \\ \hline
\end{tabular}%
\end{equation*}%
\caption{{\protect \small The numerical decay rate }$A_{2}${\protect \small \
in region }$(ii)${\protect \small \ computed from the differences tabulated
in Table \protect \ref{E2} at different }$t.$}
\label{decay_rate_ii}
\end{table}

\begin{table}[tbh]
\begin{equation*}
\begin{tabular}{|l|l|l|l|}
\hline
$t$ & $L_{4}$ at $\varepsilon =\frac{9}{80}$ & $L_{4}$ at $\varepsilon =%
\frac{14}{80}$ & $L_{4}$ at $\varepsilon =\frac{19}{80}$ \\ \hline
$40$ & $1.7974$ & $2.337$ & $2.8417$ \\ \hline
$80$ & $1.9698$ & $2.8516$ & $3.7855$ \\ \hline
$120$ & $2.2839$ & $3.3947$ & $4.4658$ \\ \hline
$160$ & $2.6081$ & $3.9556$ & $5.53$ \\ \hline
$200$ & $2.9329$ & $4.4956$ & $6.1938$ \\ \hline
$240$ & $3.2554$ & $5.0329$ & $7.2107$ \\ \hline
$280$ & $3.575$ & $5.5675$ & $8.0127$ \\ \hline
$320$ & $3.8914$ & $6.0999$ & $9.2077$ \\ \hline
\end{tabular}%
\end{equation*}%
\caption{{\protect \small The numerical decay rate }$L_{4}${\protect \small \
in region }$(iv)${\protect \small \ computed from the differences tabulated
in Table \protect \ref{E2} at different }$t.$}
\label{decay_rate_iv}
\end{table}

Moreover, we list in Tables \ref{decay_rate_ii} and \ref{decay_rate_iv}
different values of the numerical decay powers $A_{2}$ and $L_{4}$ predicted
at different $\varepsilon $ in regions $(ii)$ and $(iv)$, respectively. From
these tables, the following conclusions can be drawn:

(1) In a single region, for a smaller $\varepsilon ,$ the difference of the
norms becomes larger. A smaller $\varepsilon $ means a larger region ( i.e.,
a higher precision) for the approximation under consideration. Therefore the
time needed for the asymptotic solution to match the numerical solution is
longer.\smallskip

(2) For $\varepsilon =\frac{9}{80}$, we compare the decay estimates $%
E_{2}^{\sup }\left( t\right) $ and $E_{4}^{\sup }\left( t\right) $ at the
large time, say at $t=320.$ Then it follows that ${\small E}_{2}^{\sup
}\left( 320\right) >{\small E}_{4}^{\sup }\left( 320\right) .$ This
important message sheds light on the decay sup norm difference in the region
$(iv)$, which is smaller than that in the region $(ii).$ This result matches
the theoretical theorem obtained in \cite{MKST}, where the decay power $%
\alpha $ in region $(ii)$ satisfies $\alpha \in \left( \frac{1}{2},1\right) $
provided that the decay power $l$ in region $(iv)$ satisfies $l\geq 5.$ This
means that the decay order in region $(iv)$ is larger than that in region $%
(ii)$.

\section{Concluding remarks \label{sec:concluding}}

In this article we are aimed to get the long-time asymptotics for the CH
equation (\ref{eq:CH}). To this end, a new analytical initial condition (\ref%
{ini15}) of the CH equation is constructed first underlying the scattering
theory. It is worthy to note that this newly derived initial condition has a
non-zero reflection coefficient. In the second part of this study, a new
high-order finite difference scheme is developed and we apply it to solve
the CH solution under the specified initial condition (\ref{ini15}). Unlike
the works of \cite{MKST,bib:Monvel(2010)}, we quantitatively determine the
asymptotic time $T\ $by considering the $L^{2}$ norm differences between the
numerical and the asymptotic solutions. Beyond $T$ the predicted solutions
approach the so called long-time asymptotic solutions in the six distinct
regions (the four regions $(i)$ - $(iv)$ and the two transition regions)
under current investigation. The value of $T$ depends strongly on how large
the $\left( x,t>0\right) $ plane is required to get a good match between the
numerical and asymptotic solutions.

Under the non-reflectionless condition, we have numerically revisited the
oscillatory behaviors in the first and second oscillatory\ regions.
Moreover, we can clearly see that the investigated spatial domain contains
different asymptotic regions. The difference between the numerical and
asymptotic solutions is also intensively examined. Moreover, the rates of
convergence in different solution regions have been obtained as well. The
decay orders of the time-decay estimates between the numerical and
asymptotic solution in \cite{MKST} are computed numerically by considering
the sup norm difference between them. We have also computed the sup norm
differences at different $\varepsilon $, (i.e., at different amounts of
area) in each region. For a larger $\varepsilon ,$ the sup norm difference
becomes smaller and the decay orders numerically computed will be larger.
These numerical works are new to the best of authors' knowledge.

It can be expected that as long as the specified initial solution of the
integrable systems (including the CH equation, which is the example in this
article) can be constructed, the corresponding time evolving solutions can
be found by applying the currently developed numerical methods. Through this
study, not only the long time behavior of the CH solution but also the short
time solution profile predicted from the finite difference method can be
obtained.

\appendix

\section{Appendix}

\subsection{Previous results}

\begin{proposition}
\label{4 regions}(Theorem 2.1 in \cite{MKST}) Let $u\left( x,t\right) $ be
the solution of (\ref{eq:CH}), (\ref{eq:bc1_}). There are four sectors in
the $\left( x,t>0\right) $ half-plane. The leading term of the long-time
asymptotics of $u\left( x,t\right) $ in each sector behaves differently
depending on the magnitude of $\zeta =\frac{x}{t}:$ (see Figure \ref%
{fourregionMKST}).

\noindent (i) The soliton region: $\zeta >2+\varepsilon $ ($\varepsilon >0$
small). Let%
\begin{equation}
\alpha _{j}\left( y,t\right) =\frac{\gamma _{j}^{2}}{2\mu _{j}}e^{-2\mu
_{j}\left( y-c_{j}t\right) },\text{ where }c_{j}=\frac{2}{1-4\mu _{j}^{2}},
\label{1-solitonFOUMULA1}
\end{equation}%
with $j=1,...,N$ for some $N\in
%TCIMACRO{\U{2115} }%
%BeginExpansion
\mathbb{N}
%EndExpansion
.$ Let $\varepsilon >0$ be a small enough magnitude such that $\left[
c_{j}-\varepsilon ,c_{j}+\varepsilon \right] $ and $\left[ c_{j}-\varepsilon
,c_{j}+\varepsilon \right] $ ($i,j=1,...,N,$ $i\neq j$) are disjoint and are
inside of $\left( 2,\infty \right) .$

\noindent (i$_{1}$) If $\left \vert \zeta -c_{j}\right \vert <\varepsilon $
for some $j,$ $u\left( x,t\right) $ behaves as the sum of 1-solitons. That
is,%
\begin{equation*}
u\left( x,t\right) =u_{\text{sol}}^{j}\left( x,t\right) +O\left(
t^{-l}\right) \text{ for any }l>0,
\end{equation*}%
where $u_{\text{sol}}^{j}\left( x,t\right) $ is the 1-soliton and can be
expressed parametrically in the form as follows%
\begin{equation}
\begin{tabular}{l}
$u\left( y,t\right) =\frac{32\mu _{j}^{2}}{\left( 1-4\mu _{j}^{2}\right) ^{2}%
}\frac{\alpha _{j}\left( y,t\right) }{\left( 1+\alpha _{j}\left( y,t\right)
\right) ^{2}+\frac{16\mu _{j}^{2}}{1-4\mu _{j}^{2}}\alpha _{j}\left(
y,t\right) },$ $\  \ x\left( y,t\right) =y+\log \frac{1+\alpha _{j}\left(
y,t\right) \frac{1+2\mu _{j}}{1-2\mu _{j}}}{1+\alpha _{j}\left( y,t\right)
\frac{1-2\mu _{j}}{1+2\mu _{j}}}$.%
\end{tabular}
\label{soliton}
\end{equation}%
In the above, $\mu _{j}$ is one of the discrete eigenvalues with the
corresponding normalization constant $\gamma _{j},$ $j=1,2,...,N$,

\noindent (i$_{2}$) If $\left \vert \zeta -c_{j}\right \vert \geq
\varepsilon $ for all $j,$ $u\left( x,t\right) $ is rapidly decreasing;

\noindent (ii) The \textquotedblleft first oscillatory\textquotedblright \
region: if $0\leq \zeta <2-\varepsilon $ for any $\varepsilon >0,$ $u\left(
x,t\right) $ satisfies (\ref{oscillatory region (ii)}). That is,%
\begin{equation}
\begin{tabular}{l}
$u\left( x,t\right) =-\sqrt{\frac{2k_{0}\left( \zeta \right) \nu _{0}\left(
\zeta \right) }{\left( \frac{1}{4}+k_{0}^{2}\left( \zeta \right) \right)
\left( \frac{3}{4}-k_{0}^{2}\left( \zeta \right) \right) t}}\sin \left(
\frac{2k_{0}^{3}\left( \zeta \right) }{\left( \frac{1}{4}+k_{0}^{2}\left(
\zeta \right) \right) ^{2}}t-\nu _{0}\left( \zeta \right) \log t+\delta
_{0}\left( \zeta \right) \right) +O\left( t^{-\alpha }\right) $%
\end{tabular}
\label{ii asymp form}
\end{equation}%
for any $\alpha \in \left( \frac{1}{2},1\right) $ provided $l\geq 5$. The
notations $k_{0},\nu _{0}$ and $\delta _{0}$ in this formula can be seen as
well in \cite{MKST} and their $q_{0}$-dependence can be seen in Subsection
A.2;

\noindent (iii) The \textquotedblleft second oscillatory\textquotedblright \
region: if $\frac{-1}{4}+\varepsilon <\zeta <0$ for any $\varepsilon >0,$ $%
u\left( x,t\right) $ satisfies (\ref{2nd asymp}). That is,%
\begin{equation}
\begin{tabular}{ll}
$u\left( x,t\right) =$ & $-\sqrt{\frac{2k_{0}\left( \zeta \right) \nu
_{0}\left( \zeta \right) }{\left( \frac{1}{4}+k_{0}^{2}\left( \zeta \right)
\right) \left( \frac{3}{4}-k_{0}^{2}\left( \zeta \right) \right) t}}\sin
\left( \frac{2k_{0}^{3}\left( \zeta \right) }{\left( \frac{1}{4}%
+k_{0}^{2}\left( \zeta \right) \right) ^{2}}t-\nu _{0}\left( \zeta \right)
\log t+\bar{\delta}_{0}\left( \zeta \right) \right) $ \\
& $-\sqrt{\frac{2k_{1}\left( \zeta \right) \nu _{1}\left( \zeta \right) }{%
\left( \frac{1}{4}+k_{1}^{2}\left( \zeta \right) \right) \left(
k_{1}^{2}\left( \zeta \right) -\frac{3}{4}\right) t}}\sin \left( \frac{%
2k_{1}^{3}\left( \zeta \right) }{\left( \frac{1}{4}+k_{1}^{2}\left( \zeta
\right) \right) ^{2}}t+\nu _{1}\left( \zeta \right) \log t-\delta _{1}\left(
\zeta \right) \right) +O\left( t^{-\alpha }\right) $%
\end{tabular}
\label{iii asymp form}
\end{equation}%
for any $\alpha \in \left( \frac{1}{2},1\right) $ provided $l\geq 5$. The
notations $k_{j},\nu _{j}$ ($j=0,1$), $\bar{\delta}_{0}$ and $\delta _{1}$
in this formula can be seen in \cite{MKST}. Their $q_{0}$-dependence can be
also seen in Subsection A.2;

\noindent (iv) The \textquotedblleft fast decay\textquotedblright \ region:
if $\zeta <\frac{-1}{4}-\varepsilon $ for any $\varepsilon >0,$ $u\left(
x,t\right) $ is rapidly decreasing.
\end{proposition}

\subsection{Newly derived long-time asymptotics}

\subsubsection{Region $(i)$}

Substituting the derived scattering data $R\left( k\right) $ or (\ref%
{scatteringdata}) (i.e., $\mu _{j}=\mu _{1},$ $\gamma _{j}=\gamma _{1}$)
into (\ref{soliton}) and (\ref{1-solitonFOUMULA1}), the 1-soliton type of
solution is expressed as%
\begin{equation}
\left \{
\begin{tabular}{l}
$u\left( y,t\right) =\frac{4q_{0}^{2}}{\left( 1-q_{0}^{2}\right) ^{2}}\frac{1%
}{\exp \left( q_{0}\left( y-\frac{2}{1-q_{0}^{2}}t\right) \right) +\frac{1}{4%
}\exp \left( -q_{0}\left( y-\frac{2}{1-q_{0}^{2}}t\right) \right) +\frac{%
1+q_{0}^{2}}{1-q_{0}^{2}}},\medskip $ \\
$x\left( y,t\right) =y+\log \frac{1+\exp \left( -q_{0}\left( y-\frac{2}{%
1-q_{0}^{2}}t+\frac{1}{q_{0}}\log 2\right) \right) \frac{1+q_{0}}{1-q_{0}}}{%
1+\exp \left( -q_{0}\left( y-\frac{2}{1-q_{0}^{2}}t+\frac{1}{q_{0}}\log
2\right) \right) \frac{1-q_{0}}{1+q_{0}}}.$%
\end{tabular}%
\right.  \label{soliton wrt q}
\end{equation}

\subsubsection{Region $(ii)$}

By Proposition \ref{4 regions}, for $0\leq \zeta =\frac{x}{t}<2-\varepsilon $
the solution $u\left( x,t\right) $ behaves like (\ref{ii asymp form}), where
$k_{0}\left( \zeta \right) =\frac{1}{2}\sqrt{-\frac{1+\zeta -\sqrt{1+4\zeta }%
}{\zeta }},$ $\nu _{0}\left( \zeta \right) =-\frac{1}{2\pi }\log \left(
1-\left \vert R\left( k_{0}\left( \zeta \right) \right) \right \vert
^{2}\right) .$ Note that $\delta _{0}\left( \zeta \right) $ is expressed in
(2.13) of \cite{MKST}. Subject to the initial solution (\ref{ini15}) and the
corresponding scattering data (\ref{scatteringdata}), the asymptotic form (%
\ref{oscillatory region (ii)}) depends on the parameter $q_{0}$ in the sense
of%
\begin{equation*}
\begin{tabular}{l}
$\nu _{0}\left( \zeta \right) =-\frac{1}{2\pi }\log \frac{\sqrt{1+4\zeta }%
-1-\zeta }{\zeta q_{0}^{2}+\sqrt{1+4\zeta }-1-\zeta },$%
\end{tabular}%
\end{equation*}%
and%
\begin{equation}
\begin{tabular}{ll}
$\delta _{0}\left( \zeta \right) $ & $=\frac{\pi }{4}-\left( \tan
^{-1}\left( \frac{-2k_{0}\left( \zeta \right) }{q_{0}}\right) +\pi \right)
-\nu _{0}\left( \zeta \right) \log \frac{32\zeta \left( \sqrt{1+4\zeta }%
-1-\zeta \right) \left( 1+4\zeta -\sqrt{1+4\zeta }\right) }{\left( \sqrt{%
1+4\zeta }-1\right) ^{3}}\smallskip $ \\
& $+4\tan ^{-1}\left( \frac{q_{0}}{2k_{0}\left( \zeta \right) }\right)
+4k_{0}\left( \zeta \right) \log \frac{1+q_{0}}{1-q_{0}}+\frac{4k_{0}\left(
\zeta \right) }{\pi }\int_{-k_{0}\left( \zeta \right) }^{k_{0}\left( \zeta
\right) }\frac{1}{1+4\xi ^{2}}\log \frac{4\xi ^{2}}{q_{0}^{2}+4\xi ^{2}}d\xi
\smallskip $ \\
& $+\frac{1}{\pi }\int_{-k_{0}\left( \zeta \right) }^{k_{0}\left( \zeta
\right) }\frac{2q_{0}^{2}}{\left( q_{0}^{2}+4\xi ^{2}\right) \xi }\log
\left( k_{0}\left( \zeta \right) -\xi \right) d\xi +\arg \Gamma \left( i\nu
_{0}\left( \zeta \right) \right) .$%
\end{tabular}
\label{delta}
\end{equation}%
By using the formula for the Gamma function $\Gamma $ \cite{handbook}%
\begin{equation}
\begin{tabular}{l}
$\arg \Gamma \left( z+1\right) =\arg \Gamma \left( z\right) +\tan ^{-1}\frac{%
y}{x},$ $z=x+iy,\smallskip $ \\
$\arg \Gamma \left( x+iy\right) =y\frac{\Gamma ^{\prime }\left( x\right) }{%
\Gamma \left( x\right) }+\sum_{n=0}^{\infty }\left( \frac{y}{x+n}-\tan ^{-1}%
\frac{y}{x+n}\right) ,$ $x+iy\neq 0,-1,-2,...$%
\end{tabular}
\label{Gamma function formula}
\end{equation}%
the last term $\arg \Gamma \left( i\nu _{0}\left( \zeta \right) \right) $ in
(\ref{delta}) can be expressed as%
\begin{equation*}
\begin{tabular}{ll}
$\arg \Gamma \left( i\nu _{0}\left( \zeta \right) \right) $ & $=\arg \Gamma
\left( 1+i\nu _{0}\left( \zeta \right) \right) -\frac{\pi }{2}$ \\
& $=\nu _{0}\frac{\Gamma ^{\prime }\left( 1\right) }{\Gamma \left( 1\right) }%
+\sum_{n=0}^{\infty }\left( \frac{\nu _{0}}{1+n}-\tan ^{-1}\frac{\nu _{0}}{%
1+n}\right) -\frac{\pi }{2}$ \\
& $=\left( -\gamma \nu _{0}\left( \zeta \right) +\sum_{n=0}^{\infty }\left(
\frac{\nu _{0}\left( \zeta \right) }{1+n}-\tan ^{-1}\frac{\nu _{0}\left(
\zeta \right) }{1+n}\right) \right) -\frac{\pi }{2}$%
\end{tabular}%
\end{equation*}%
where $\gamma =\lim_{n\rightarrow \infty }\left( \sum_{k=1}^{n}\frac{1}{k}%
-\log n\right) =0.5772156649....$ is the Euler constant.

\subsubsection{Region $(iii)$}

By Proposition \ref{4 regions}, for $\frac{-1}{4}+\varepsilon \leq \zeta =%
\frac{x}{t}<0,$ $u\left( x,t\right) $ takes the form like (\ref{iii asymp
form}), where%
\begin{equation*}
\begin{tabular}{ll}
$k_{0}\left( \zeta \right) =\frac{1}{2}\sqrt{-\frac{1+\zeta -\sqrt{1+4\zeta }%
}{\zeta }},$ & $k_{1}\left( \zeta \right) =\frac{1}{2}\sqrt{-\frac{1+\zeta +%
\sqrt{1+4\zeta }}{\zeta }},$%
\end{tabular}%
\end{equation*}%
\begin{equation*}
\begin{tabular}{l}
$\nu _{l}\left( \zeta \right) =-\frac{1}{2\pi }\log \left( 1-\left \vert
R\left( k_{l}\left( \zeta \right) \right) \right \vert ^{2}\right) ,$ $%
l=0,1. $%
\end{tabular}%
\end{equation*}%
Note that $\bar{\delta}_{0}\left( \zeta \right) $ and $\delta _{1}\left(
\zeta \right) $ are expressed in (2.17) of \cite{MKST}. Similar to the
slowly decaying modulated oscillation region $(ii)$, through the initial
data (\ref{ini15}) together with the scattering data (\ref{scatteringdata})
the asymptotic form (\ref{2nd asymp}) can be rewritten in terms of the
parameter $q_{0}$ as
\begin{equation*}
\begin{tabular}{ll}
$\nu _{0}\left( \zeta \right) =-\frac{1}{2\pi }\log \frac{\sqrt{1+4\zeta }%
-1-\zeta }{\zeta q_{0}^{2}+\sqrt{1+4\zeta }-1-\zeta },$ & $\nu _{1}\left(
\zeta \right) =-\frac{1}{2\pi }\log \frac{-\sqrt{1+4\zeta }-1-\zeta }{\zeta
q_{0}^{2}-\sqrt{1+4\zeta }-1-\zeta }.$%
\end{tabular}%
\end{equation*}%
Also, $\bar{\delta}_{0}\left( \zeta \right) ,$ $\delta _{1}\left( \zeta
\right) $ can be rewritten as%
\begin{equation*}
\begin{tabular}{ll}
$\bar{\delta}_{0}\left( \zeta \right) $ & $=\frac{\pi }{4}-\left( \tan
^{-1}\left( \frac{-2k_{0}\left( \zeta \right) }{q_{0}}\right) +\pi \right)
+\arg \Gamma \left( i\nu _{0}\left( \zeta \right) \right) $ \\
& $-\nu _{0}\left( \zeta \right) \log \frac{32\zeta \left( \sqrt{1+4\zeta }%
-1-\zeta \right) \left( 1+4\zeta -\sqrt{1+4\zeta }\right) }{\left( \sqrt{%
1+4\zeta }-1\right) ^{3}}$ \\
& $+4\tan ^{-1}\left( \frac{q_{0}}{2k_{0}\left( \zeta \right) }\right)
+4k_{0}\left( \zeta \right) \log \frac{1+q_{0}}{1-q_{0}}\smallskip $ \\
& $+\frac{4k_{0}\left( \zeta \right) }{\pi }\left( \int_{-\infty
}^{-k_{1}\left( \zeta \right) }+\int_{-k_{0}\left( \zeta \right)
}^{k_{0}\left( \zeta \right) }+\int_{k_{1}\left( \zeta \right) }^{\infty
}\right) \frac{1}{1+4\xi ^{2}}\log \frac{4\xi ^{2}}{q_{0}^{2}+4\xi ^{2}}d\xi
\smallskip $ \\
& $+\frac{1}{\pi }I_{0}+2\nu _{1}\left( \zeta \right) \log \frac{k_{1}\left(
\zeta \right) -k_{0}\left( \zeta \right) }{k_{1}\left( \zeta \right)
+k_{0}\left( \zeta \right) }$%
\end{tabular}%
\end{equation*}%
with%
\begin{equation*}
\begin{tabular}{ll}
$I_{0}=$ & $\left( \int_{-\infty }^{-k_{1}\left( \zeta \right)
}+\int_{-k_{0}\left( \zeta \right) }^{k_{0}\left( \zeta \right) }\right)
\log \left( k_{0}\left( \zeta \right) -\xi \right) \frac{2q_{0}^{2}}{\left(
q_{0}^{2}+4\xi ^{2}\right) \xi }d\xi $ \\
& $+\int_{k_{1}\left( \zeta \right) }^{\infty }\log \left( \xi -k_{0}\left(
\zeta \right) \right) \frac{2q_{0}^{2}}{\left( q_{0}^{2}+4\xi ^{2}\right)
\xi }d\xi ,$%
\end{tabular}%
\end{equation*}%
\begin{equation*}
\begin{tabular}{ll}
$\delta _{1}\left( \zeta \right) $ & $=\frac{\pi }{4}+\tan ^{-1}\left( \frac{%
-2k_{1}\left( \zeta \right) }{q_{0}}\right) +\pi +\arg \Gamma \left( i\nu
_{1}\left( \zeta \right) \right) $ \\
& $-\nu _{1}\left( \zeta \right) \log \frac{32\zeta \left( \sqrt{1+4\zeta }%
+1+\zeta \right) \left( 1+4\zeta +\sqrt{1+4\zeta }\right) }{-\left( \sqrt{%
1+4\zeta }+1\right) ^{3}}$ \\
& $-4\tan ^{-1}\left( \frac{q_{0}}{2k_{1}\left( \zeta \right) }\right)
+4k_{1}\left( \zeta \right) \log \frac{1+q_{0}}{1-q_{0}}\smallskip $ \\
& $+\frac{4k_{1}\left( \zeta \right) }{\pi }\left( \int_{-\infty
}^{-k_{1}\left( \zeta \right) }+\int_{-k_{0}\left( \zeta \right)
}^{k_{0}\left( \zeta \right) }+\int_{k_{1}\left( \zeta \right) }^{\infty
}\right) \frac{1}{1+4\xi ^{2}}\log \frac{4\xi ^{2}}{q_{0}^{2}+4\xi ^{2}}d\xi
\smallskip $ \\
& $-\frac{1}{\pi }I_{1}-2\nu _{0}\left( \zeta \right) \log \frac{k_{1}\left(
\zeta \right) -k_{0}\left( \zeta \right) }{k_{1}\left( \zeta \right)
+k_{0}\left( \zeta \right) }.$%
\end{tabular}%
\end{equation*}%
In the above, $I_{1}$ and $\arg \Gamma \left( i\nu _{1}\left( \zeta \right)
\right) $ are expressed as follows%
\begin{equation*}
\begin{tabular}{ll}
$I_{1}=$ & $\left( \int_{-\infty }^{-k_{1}\left( \zeta \right)
}+\int_{-k_{0}\left( \zeta \right) }^{k_{0}\left( \zeta \right) }\right)
\log \left( k_{1}\left( \zeta \right) -\xi \right) \frac{2q_{0}^{2}}{\left(
q_{0}^{2}+4\xi ^{2}\right) \xi }d\xi $ \\
& $+\int_{k_{1}\left( \zeta \right) }^{\infty }\log \left( \xi -k_{1}\left(
\zeta \right) \right) \frac{2q_{0}^{2}}{\left( q_{0}^{2}+4\xi ^{2}\right)
\xi }d\xi ,$%
\end{tabular}%
\end{equation*}%
\begin{equation*}
\begin{tabular}{l}
$\arg \Gamma \left( i\nu _{1}\left( \zeta \right) \right) =-\gamma \nu
_{1}\left( \zeta \right) +\sum_{n=0}^{\infty }\left( \frac{\nu _{1}\left(
\zeta \right) }{1+n}-\tan ^{-1}\frac{\nu _{1}\left( \zeta \right) }{1+n}%
\right) -\frac{\pi }{2}.$%
\end{tabular}%
\end{equation*}

\subsection{Transition regions}

According to \cite{MKST} and \cite{bib:Monvel(2010)}, the Riemann-Hilbert
Problem (RHP) of the CH equation is known to have a unique solution $M\left(
k;y,t\right) $ sought subject to the jump condition $M_{+}\left(
k;y,t\right) =M_{-}\left( k;y,t\right) J\left( k;y,t\right) $ for $k\in
%TCIMACRO{\U{211d} }%
%BeginExpansion
\mathbb{R}
%EndExpansion
,$ where%
\begin{equation*}
J\left( k;y,t\right) =e^{-it\theta \left( k;\hat{\zeta}\right) \sigma
_{3}}\left(
\begin{array}{cc}
1-\left \vert R\left( k\right) \right \vert ^{2} & -\overline{R\left(
k\right) } \\
R\left( k\right) & 1%
\end{array}%
\right) e^{it\theta \left( k;\hat{\zeta}\right) \sigma _{3}}
\end{equation*}%
with%
\begin{equation*}
\theta (k;\hat{\zeta})=\hat{\zeta}k-\frac{2k}{1+4k^{2}},\text{ \ }\sigma
_{3}=\left(
\begin{array}{cc}
1 & 0 \\
0 & 1%
\end{array}%
\right) ,\text{ }\hat{\zeta}=\frac{y}{t}
\end{equation*}

Following the method of Deift and Zhou \cite{DeiftZhou}, Boutet de Monvel et
al. analyzed the long-time asymptotics of the RHP according to the sign of $%
\theta (k;\hat{\zeta})$ in different regions in the spectral domain $k$ such
that the limit of $\theta (k;\hat{\zeta})$ can be observed as $t\rightarrow
\infty .$ There are four cases (see Figure 4.1 in \cite{MKST}): $\hat{\zeta}%
>2,$ $0<\hat{\zeta}<2,$ $-\frac{1}{4}<\hat{\zeta}<0$ and $\hat{\zeta}<\frac{%
-1}{4}$ under investigation. In the regions $0<\hat{\zeta}<2-\varepsilon $
and $-\frac{1}{4}+\varepsilon <\hat{\zeta}<0$ for any fixed $\varepsilon >0$%
, the stationary phase points (the points $k$ at which $\frac{d\theta \left(
k;\hat{\zeta}\right) }{dk}=0$), namely,%
\begin{equation*}
\begin{tabular}{ll}
$\pm k_{0}(\hat{\zeta})=\pm \frac{1}{2}\sqrt{-\frac{1+\hat{\zeta}-\sqrt{1+4%
\hat{\zeta}}}{\hat{\zeta}}},$ & $\pm k_{1}(\hat{\zeta})=\pm \frac{1}{2}\sqrt{%
-\frac{1+\hat{\zeta}+\sqrt{1+4\hat{\zeta}}}{\hat{\zeta}}},$%
\end{tabular}%
\end{equation*}%
are well separated. Also, the two oscillatory asymptotics can be obtained by
considering the long-time asymptotics in a small neighborhood of $\pm k_{0}(%
\hat{\zeta})$ or $\pm k_{1}(\hat{\zeta})$ in the spectral domain $k.$

As long as $\hat{\zeta}\rightarrow 2$ and $\hat{\zeta}\rightarrow -\frac{1}{4%
}$ as $t\rightarrow \infty ,$ $k_{0}(\hat{\zeta})\ $and $-k_{0}(\hat{\zeta}%
)\ $approach $0$ and $\pm k_{0}(\hat{\zeta}),$ $\pm k_{1}(\hat{\zeta}%
)\rightarrow \pm \frac{\sqrt{3}}{2},$ respectively. We need therefore a more
detailed analysis of the asymptotics. Consider, for example, the first
transition region. If $\hat{\zeta}\rightarrow 2$, it follows that%
\begin{equation*}
\theta (k;\hat{\zeta})=\frac{4}{3}\hat{k}^{3}+s\hat{k}+O\left( \hat{k}^{5}t^{%
\frac{-2}{3}}\right)
\end{equation*}%
where $\hat{k}=\left( 6t\right) ^{\frac{1}{3}}k,$ $s=6^{\frac{-1}{3}}(\hat{%
\zeta}-2)t^{\frac{2}{3}}.$ For the variable $s$ in a bounded region, say $%
\left \vert \hat{\zeta}-2\right \vert t^{\frac{2}{3}}<C$ for any $C>0,$ the
RHP for the CH equation can be asymptotically approximated by the RHP of the
Painlev\'{e} equation of type II (see \cite{FIKN}).

\begin{proposition}
\label{Painleve region}(Theorem 1.1 in \cite{bib:Monvel(2010)}) Let $u\left(
x,t\right) $ be the solution of (\ref{eq:CH}), (\ref{eq:bc1_}).

(a) For $\left \vert \frac{x}{t}-2\right \vert t^{\frac{2}{3}}<C,$ $u\left(
x,t\right) $ satisfies (\ref{Painleve 1st reg}) where $s=6^{\frac{-1}{3}%
}\left( \frac{x}{t}-2\right) t^{\frac{2}{3}},$ and $v=v\left( s\right)
=v\left( s;-R\left( 0\right) \right) $ is the real valued, non-singular
solution of the $P_{II}$ equation (\ref{P2}) fixed by $v\left( s\right) \sim
-R\left( 0\right) $Ai$\left( s\right) $ $s\rightarrow \infty .$

(b) For $\left \vert \frac{x}{t}+\frac{1}{4}\right \vert t^{\frac{2}{3}}<C,$
$u\left( x,t\right) $ satisfies (\ref{Painleve 2nd reg}) where%
\begin{equation*}
s_{1}=-\left( \frac{16}{3}\right) ^{\frac{1}{3}}\left( \frac{x}{t}+\frac{1}{4%
}\right) t^{\frac{2}{3}},\text{ }\psi \left( s_{1},t\right) =\frac{-3\sqrt{3}%
}{4}t-\frac{3^{\frac{5}{6}}}{2^{\frac{4}{3}}}s_{1}t^{\frac{1}{3}}+\triangle
_{1},
\end{equation*}%
with $\triangle _{1}$ being expressed as (1.6) in \cite{bib:Monvel(2010)}.
Here, $v_{1}\left( s\right) =v_{1}\left( s;\left \vert R\left( \frac{\sqrt{3}%
}{2}\right) \right \vert \right) $ is the real valued, non-singular solution
of (\ref{P2}) fixed by $v_{1}\left( s\right) \sim \left \vert R\left( \frac{%
\sqrt{3}}{2}\right) \right \vert $Ai$\left( s\right) $ as $s\rightarrow
\infty .$
\end{proposition}

\subsubsection{Second transition region}

By considering the specified initial solution (\ref{ini15}) with the
scattering data (\ref{scatteringdata}), the asymptotic form (\ref{Painleve
2nd reg}) can be rewritten in a form that depends on the parameter $q_{0}$.
That is, $\triangle _{1}=\triangle _{1}\left( q_{0}\right) $ in the
following form:%
\begin{equation*}
\begin{tabular}{ll}
$\triangle _{1}$ & $=\frac{-4\sqrt{3}}{\pi }\int_{0}^{\infty }\frac{1}{%
1+4\xi ^{2}}\log \frac{4\xi ^{2}}{q_{0}^{2}+4\xi ^{2}}d\xi +\tan ^{-1}\left(
\frac{-\sqrt{3}}{q_{0}}\right) +\pi -4\tan ^{-1}\left( \frac{q_{0}}{\sqrt{3}}%
\right) \smallskip $ \\
& $-2\sqrt{3}\log \frac{1+q_{0}}{1-q_{0}}+\frac{1}{\pi }\int_{-\infty
}^{\infty }\frac{1}{\xi -\frac{\sqrt{3}}{2}}\log \frac{4\xi ^{2}}{%
q_{0}^{2}+4\xi ^{2}}d\xi .$%
\end{tabular}%
\end{equation*}

In the following we concentrate on the connection formula (\ref{Clarkson
Mcleod formula}) of (\ref{P2}). By \cite{ClarksonMcleod}, we can have%
\begin{equation*}
\begin{tabular}{l}
$d^{2}=\frac{-1}{\pi }\ln \left( 1-\left \vert R\left( \frac{\sqrt{3}}{2}%
\right) \right \vert ^{2}\right) ,\smallskip $ \\
$\theta _{0}=\frac{3}{2}d^{2}\ln 2+\arg \left[ \Gamma \left( 1-\frac{1}{2}%
d^{2}i\right) \right] -\frac{1}{4}\pi .$%
\end{tabular}%
\end{equation*}%
Note that $\left \vert R\left( \frac{\sqrt{3}}{2}\right) \right \vert =\frac{%
q_{0}}{\sqrt{q_{0}^{2}+3}},$ $\left \vert R\left( \frac{\sqrt{3}}{2}\right)
\right \vert $ is equal to $\frac{1}{\sqrt{13}}$ if $q_{0}=\frac{1}{2}.$ By
the formulae (\ref{Gamma function formula}), it follows that%
\begin{equation*}
\arg \left[ \Gamma \left( 1-\frac{1}{2}d^{2}i\right) \right] =\frac{1}{2}%
d^{2}\gamma +\sum_{n=0}^{\infty }\left( \frac{\left( -\frac{1}{2}%
d^{2}\right) }{1+n}-\tan ^{-1}\frac{\left( -\frac{1}{2}d^{2}\right) }{1+n}%
\right) .
\end{equation*}%
Hence, the following results can be obtained%
\begin{equation*}
\begin{tabular}{l}
$d^{2}=\frac{-1}{\pi }\ln \left( 1-r^{2}\right) =\frac{-1}{\pi }\ln \frac{12%
}{13}=\frac{1}{\pi }\ln \frac{13}{12},$ \\
$\theta _{0}=\frac{3}{2}d^{2}\ln 2+\arg \left[ \Gamma \left( 1-\frac{1}{2}%
d^{2}i\right) \right] -\frac{1}{4}\pi $ \\
$=\left( \frac{1}{\pi }\ln \frac{13}{12}\right) \left( \frac{3}{2}\ln 2+%
\frac{1}{2}\gamma \right) +\sum_{n=0}^{\infty }\left( \frac{\left( -\frac{1}{%
2}\left( \frac{1}{\pi }\ln \frac{13}{12}\right) \right) }{1+n}-\tan ^{-1}%
\frac{\left( -\frac{1}{2}\left( \frac{1}{\pi }\ln \frac{13}{12}\right)
\right) }{1+n}\right) -\frac{1}{4}\pi .$%
\end{tabular}%
\end{equation*}%
The connection formula for $v\left( s\right) $ can be therefore expressed as

\begin{eqnarray}
v\left( s\right) &\sim &\frac{1}{\sqrt{13}}\frac{1}{2\sqrt{\pi }}s^{\frac{-1%
}{4}}e^{\frac{-2}{3}s^{\frac{-3}{2}}}:=B_{+}\left( s\right) ,\text{ \ }%
s\rightarrow \infty ,  \label{connection formula} \\
v\left( s\right) &\sim &\sqrt{\frac{1}{\pi }\ln \frac{13}{12}}\left \vert
s\right \vert ^{\frac{-1}{4}}\sin \left \{ \frac{2}{3}\left \vert s\right
\vert ^{\frac{3}{2}}-\frac{3}{4}\left( \frac{1}{\pi }\ln \frac{13}{12}%
\right) \ln \left \vert s\right \vert -\theta _{0}\right \} :=B_{-}\left(
s\right) ,\text{ \ }s\rightarrow -\infty .  \notag
\end{eqnarray}

\section*{Acknowledgements}

%%%%%%%%%%%%%%%%%%%%%%%%%%%%%%%%%%%%%%%%%%%%%%%%%%%%%%%%%%%%%%%%%%%%%%%%%%%%%%%%%%%%%%%%%%%
This research is supported by the Ministry of Science and Technology (MOST)
under the Grants NSC 98-2628-M-002-006 and 98-2811-E-002-006. The last
author would like to thank Prof. Chang-Shou Lin for his long-term
encouragement of conducting this study and the fruitful discussion in the
past two years.

\end{document}